\documentclass[11pt,oneside,reqno]{article}
\usepackage[margin=1.3in]{geometry}

\usepackage{amsmath,amsthm,amssymb,color}
\usepackage[authoryear]{natbib}
\usepackage{booktabs}

\usepackage{ulem}
\usepackage{bbm}
\usepackage{mathrsfs}
\usepackage[breaklinks=true]{hyperref}
\usepackage{graphicx}
\usepackage{tikz}
\usetikzlibrary{arrows, automata}
\usetikzlibrary{calc}
\usetikzlibrary{positioning}

\numberwithin{equation}{section}
\allowdisplaybreaks[4]

\theoremstyle{plain}
\newtheorem{theorem}{Theorem}[section]
\newtheorem{proposition}{Proposition}[section]
\newtheorem{lemma}{Lemma}[section]
\newtheorem{corollary}{Corollary}[section]

\theoremstyle{definition}
\newtheorem{definition}{Definition}[section]

\newtheorem{remark}{Remark}[section]

\makeatletter

\newcount\minute
\newcount\hour
\newcount\hourMins
\def\now{%
\minute=\time%
\hour=\time \divide \hour by 60%
\hourMins=\hour \multiply\hourMins by 60%
\advance\minute by -\hourMins%
\zeroPadTwo{\the\hour}:\zeroPadTwo{\the\minute}%
}
\def\zeroPadTwo#1{\ifnum #1<10 0\fi#1}

\renewcommand{\cite}{\citet}

\def\^#1{\ifmmode {\mathaccent"705E #1} \else {\accent94 #1} \fi}
\def\~#1{\ifmmode {\mathaccent"707E #1} \else {\accent"7E #1} \fi}

\def\*#1{#1^\ast}
\edef\-#1{\noexpand\ifmmode {\noexpand\bar{#1}} \noexpand\else \-#1\noexpand\fi}
\def\>#1{\vec{#1}}
\def\.#1{\dot{#1}}

\def\wt#1{\widetilde{#1}}

\def\atop{\@@atop}
\def\*#1{\mathscr{#1}}

\renewcommand{\leq}{\leqslant}
\renewcommand{\geq}{\geqslant}

\newcommand{\eq}{\eqref}

\newcommand{\IE}{\mathbbm{E}}

\newcommand{\Var}{\mathop{\mathrm{Var}}\nolimits}
\newcommand{\Cov}{\mathop{\mathrm{Cov}}}

\def\be#1{\begin{equation*}#1\end{equation*}}
\def\ben#1{\begin{equation}#1\end{equation}}
\def\bes#1{\begin{equation*}\begin{split}#1\end{split}\end{equation*}}
\def\besn#1{\begin{equation}\begin{split}#1\end{split}\end{equation}}

\def\ba#1{\begin{align*}#1\end{align*}}

\newcommand{\mcl}[1]{\mathcal{#1}}

\def\norm#1{\Vert#1\Vert}

\def\mid{\vert}

\def\beqn#1\eeqn{\begin{align}#1\end{align}}
\def\beq#1\eeq{\begin{align*}#1\end{align*}}

\usepackage{graphicx}
\usepackage{latexsym}
\usepackage{amsmath,amsthm,amssymb,amscd}
\usepackage{epsf,amsmath}

\def\E{{\IE}}

\newcommand{\ul}[1]{\underline{#1}}
\newcommand{\ol}[1]{\overline{#1}}

\renewcommand\section{\@startsection {section}{1}{\z@}%
{-3.5ex \@plus -1ex \@minus -.2ex}%
{1.3ex \@plus.2ex}%
{\center\small\sc\mathversion{bold}\MakeUppercase}}

\def\subsection#1{\@startsection {subsection}{2}{0pt}%
{-3.5ex \@plus -1ex \@minus -.2ex}%
{1ex \@plus.2ex}%
{\bf\mathversion{bold}}{#1}}

\def\subsubsection#1{\@startsection{subsubsection}{3}{0pt}%
{\medskipamount}%
{-10pt}%
{\normalsize\itshape}{\kern-2.2ex. #1.}}

\def\blfootnote{\xdef\@thefnmark{}\@footnotetext}

\makeatother

\begin{document}

\title{High-dimensional Central Limit Theorems by Stein's Method}
\author{Xiao Fang and Yuta Koike}
\date{\it The Chinese University of Hong Kong and The University of Tokyo} 
\maketitle

\noindent{\bf Abstract:} 
We obtain explicit error bounds for the $d$-dimensional normal approximation on hyperrectangles for a random vector that has a Stein kernel, or admits an exchangeable pair coupling, or is a non-linear statistic of independent random variables or a sum of $n$ locally dependent random vectors.
We assume the approximating normal distribution has a non-singular covariance matrix.
The error bounds vanish even when the dimension $d$ is much larger than the sample size $n$.
We prove our main results using the approach of G\"otze (1991) in Stein's method, together with modifications of an estimate of Anderson, Hall and Titterington (1998) and a smoothing inequality of Bhattacharya and Rao (1976). 
For sums of $n$ independent and identically distributed isotropic random vectors having a log-concave density, we obtain an error bound that is optimal up to a $\log n$ factor.
We also discuss an application to multiple Wiener-It\^{o}integrals.

\medskip



\section{Introduction and Main Results}\label{sec1}

Motivated by modern statistical applications in large-scale data, there has been a recent wave of interest in proving high-dimensional central limit theorems. 
Starting from the pioneering work by \cite{CCK13}, who established a Gaussian approximation for maxima of sums of centered independent random vectors, many articles have been devoted to the development of this subject: For example, see \cite{CCK17,CCKK19} for generalization to normal approximation on hyperrectangles and improvements of the error bound, \cite{Ch18,ChKa19,SoChKa19} for extensions to $U$-statistics, \cite{CCK19,ZhCh18,ZhWu17} for sums of dependent random vectors, and \cite{BCCHK18} for a general survey and statistical applications.
In particular, for $W=n^{-1/2}\sum_{i=1}^n X_i$ where $\{X_1,\dots, X_n\}$ are centered independent random vectors in $\mathbb{R}^d$ and satisfy certain regularity conditions, \cite{CCKK19} proved that
\ben{\label{f11}
\sup_{h=1_A: A\in \mathcal{R}} |E h(W)- E h(Z)|\leq C_0\left(\frac{\log^5(dn)}{n} \right)^{1/4},
}
where $\mathcal{R}:=\{\Pi_{j=1}^d (a_j, b_j), -\infty\leq a_j\leq b_j\leq \infty\}$, $Z$ is a centered Gaussian vector with the same covariance matrix as $W$ and $C_0$ is a positive constant that is independent of $d$ and $n$. 

The distance between two probability measures on $\mathbb{R}^d$ considered in \eq{f11} is stronger than the multivariate Kolmogorov distance. The error bound vanishes if $\log d=o(n^{1/5})$,
which allows $d$ to be much larger than $n$. The result in \eq{f11} is useful in many statistical applications in high-dimensional inference such as construction of simultaneous confidence intervals and strong control of family-wise error rate in multiple testing; see \cite{BCCHK18} for details. 
In the literature, people have also considered bounding other (stronger) distances in multivariate normal approximations. However, they typically require $d$ to be sub-linear in $n$. 
We discuss some of the recent results in Section \ref{sec1.1} below.


To date, the proofs for results such as \eq{f11} in the literature all involve smoothing the maximum function $\max_{1\leq j\leq d} x_j$ by $\frac{1}{\beta}\log \sum_{j=1}^d e^{\beta x_j}$ for a large $\beta$ (cf. Theorem 1.3 of \cite{Ch05}).
In this paper, we use a new method to prove high-dimensional normal approximations on hyperrectangles.
We assume the approximating normal distribution has a non-singular covariance matrix.
Our method combines the approach of \cite{Go91} in Stein's method with modifications of an estimate of \cite{AnHaTi98} and a smoothing inequality of \cite{BhRa76}. 
We improve the bound in \eq{f11} to $C_0\big(\frac{\log^4 (d n) }{n}\big)^{1/3}$ when the smallest eigenvalue of $\Cov(W)$ is bounded away from 0 by a constant independent of $d$ and $n$ (cf. Corollary~\ref{c2} below). 
We further improve the bound to $C_0\big(\frac{\log^3 d}{n}\big)^{1/2}\log n$, which is optimal up to the $\log n$ factor, for sums of independent and identically distributed (i.i.d.) isotropic random vectors with log-concave distributions (cf. Corollary~\ref{c0} below).
Moreover, our method works for general dependent random vectors and we state our main results for $W$ that has a Stein kernel, or admits an exchangeable pair coupling, or is a non-linear statistic of independent random variables or a sum of locally dependent random vectors. 
We prove our main results in Section~\ref{sec2}. We also discuss an application to multiple Wiener-It\^{o}integrals. 
Some details are deferred to an appendix.

Throughout the paper, we always assume $d\geq3$ so that $\log d>1$. 
Also, $W$ denotes a random vector in $\mathbb{R}^d$ with $E W=0$. 
We use $Z\sim N(0, \Sigma)$ to denote a $d$-dimensional Gaussian variable with covariance matrix $\Sigma=(\Sigma_{jk})_{1\leq j, k\leq d}$ and denote
\besn{\label{eq:sigma}
&\overline \sigma^2:=\overline \sigma^2(\Sigma)=\max_{1\leq j\leq d}\Sigma_{jj}, \\
& \underline \sigma^2:=\underline \sigma^2(\Sigma)=\min_{1\leq j\leq d}\Sigma_{jj}, \\
&\sigma_*^2:=\sigma_*^2(\Sigma)=\text{smallest eigenvalue of $\Sigma$}.
} 
Note that in the isotropic case $\Sigma=I_d$, $\overline \sigma^2=\underline \sigma^2=\sigma_*^2=1$.
We use $C$ to denote positive absolute constants, which may differ in different expressions.
We use $\partial_j f, \partial_{jk} f,$ etc to denote partial derivatives. 
For an $\mathbb{R}^d$-vector $w$, we use $w_j, 1\leq j\leq d$ to denote its components and write $\|w\|_\infty=\max_{1\leq j\leq d}|w_j|$.

We first consider random vectors that have a Stein kernel, which was defined in \cite{LeNoPe15} and used implicitly in, for example, \cite{Ch09} and \cite{NoPe09} (see also Lecture VI of \cite{St86}). 

\begin{definition}[Stein kernel]
A $d\times d$ matrix-valued measurable function $\tau^W=(\tau^W_{ij})_{1\leq i,j\leq d}$ on $\mathbb{R}^d$ is called a \textit{Stein kernel} for (the law of) $W$ if $E|\tau^W_{ij}(W)|<\infty$ for any $i,j\in\{1,\dots,d\}$ and 
\begin{equation*}
\sum_{j=1}^dE[\partial_jf(W)W_j]=\sum_{i,j=1}^dE[\partial_{ij}f(W)\tau^W_{ij}(W)]
\end{equation*}
for any $C^\infty$ function $f:\mathbb{R}^d\to\mathbb R$ with bounded partial derivatives of all orders.
\end{definition}

If $W$ has a Stein kernel, then in applying Stein's method, we only need to deal with the first and second derivatives of the solution to the Stein equation. In this case, we obtain the following simple bound.

\begin{theorem}[Error bound using Stein kernels]\label{t1}
Suppose that $W$ has a Stein kernel $\tau^W=(\tau^W_{jk})_{1\leq j,k\leq d}$. 
Let $Z\sim N(0, \Sigma)$.
Then we have
\ben{\label{f12}
\sup_{h=1_A: A\in \mathcal{R}} |E h(W)- E h(Z)|
\leq C \frac{\Delta_W}{\sigma_*^2} (\log d)(|\log(\frac{\underline \sigma \Delta_W}{\overline \sigma \sigma_*^2})|\vee 1),
}
where the $\sigma$'s are defined in \eq{eq:sigma} and 
\[
\Delta_W:= E\left[\max_{1\leq j,k\leq d}|\Sigma_{jk}-\tau^W_{jk}(W)|\right].
\]
\end{theorem}

\begin{remark}\label{rem1}
In practice, we typically choose $\Sigma=\Cov(W)$ (so that $E \tau^W_{jk}(W)=\Sigma_{jk}$), although it is not required in the above theorem.
Moreover, since
\be{
\sup_{h=1_A: A\in \mathcal{R}} |E h(W)- E h(Z)|=\sup_{h=1_A: A\in \mathcal{R}} |E h(MW)- E h(MZ)|
}
for any diagonal matrix $M$, we have the freedom to do component-wise scaling for $W$ so that the right-hand side of \eq{f12} is minimized.
This minimization problem seems non-trivial, except that one should obvious shrink each component of $W$ as much as possible for a given value of $\sigma_*$.
This remark applies to all the general bounds below (cf. Theorems \ref{t2}--\ref{ft5}).
For simplicity, in applications below (cf. Corollary \ref{c0}--\ref{c2}), we do the most natural component-wise scaling for $W$ so that $\Var(W_j)=1$, $1\leq j \leq d$ and choose $\Sigma=\Cov(W)$. As a result, $\overline \sigma=\underline \sigma=1$ and only $1/\sigma_*$ appears in the upper bound. 
This factor can be removed if $\sigma_*$ is bounded away from 0 by an absolute constant. 
We call it the \textit{strongly non-singular} case. One example is the isotropic case where $\Sigma=I_d$.
\end{remark}

\begin{remark}
\cite[Theorem 5.1]{CCKK19} proved\footnote{\eq{r06} is deduced from their result together with the proof of \cite[Corollary 5.1]{CCK17} and the Stein kernel for $(W^\top, -W^\top)^\top$.} that if $W$ has a Stein kernel $\tau^W=(\tau^W_{jk})_{1\leq j,k\leq d}$ and 
$Z\sim N(0,\Sigma)$ with the diagonal entries $\Sigma_{jj}\geq c$ for all $j=1,\dots, d$ and some constant $c>0$, then 
\ben{\label{r06}
\sup_{h=1_A: A\in \mathcal{R}} |P(W\in A)- P(Z\in A)|\leq C'  \Delta_W^{1/2} \log d,
}
where $C'$ depends only on $c$.
They also showed that the bound \eq{r06} is asymptotically sharp (personal communication). 
Theorem \ref{t1} shows that under the additional assumption that $\Sigma$ is non-singular and the ratio of the largest and the smallest diagonal entries of $\Sigma$, $\frac{\overline \sigma}{\underline \sigma}$, is bounded, the bound $\eq{r06}$ can be improved to
\be{
C_* \Delta_W \log d (|\log \Delta_W|\vee 1),
}
where $C_*$ depends only on $\sigma_*$.
Since $\Sigma$ is singular in the example attaining the upper bound in \eqref{r06} asymptotically, this improvement comes from the non-singularity assumption on $\Sigma$.
\end{remark}

%
As an illustration, we apply Theorem \ref{t1} to sums of i.i.d.~variables with log-concave densities. 
Recall that a probability measure $\mu$ on $\mathbb{R}^d$ has a log-concave density if it is supported on (the closure of) an open convex set $\Omega\subset \mathbb{R}^d$ and, on $\Omega$, it has a density of the form $e^{-V}$ with $V: \Omega\to \mathbb{R}$ a convex (hence continuous) function; see \cite{SaWe14} for more details. Note that the support of $\mu$ must be full dimensional because $\mu$ has a density. 
In this situation, under some regularity assumptions, \cite{Fathi19} provides a way to construct Stein kernels having some nice properties. This enables us to obtain the following near optimal error bound.

\begin{corollary}\label{c0}
Let $\mu$ be a probability measure on $\mathbb{R}^d$ with a log-concave density. 
Suppose $\mu$ has mean 0 and a covariance matrix $\Sigma$ with diagonal entries all equal to 1 and smallest eigenvalue $\sigma_*^2>0$. 
Let $W=n^{-1/2}\sum_{i=1}^n X_i\in \mathbb{R}^d$ with $n\geq 3$, where $\{X_1,\dots, X_n\}$ are i.i.d.~with law $\mu$. 
Let $Z\sim N(0,\Sigma)$.
Then 
\[
\sup_{h=1_A: A\in \mathcal{R}} |E h(W)- E h(Z)|
\leq \frac{C}{\sigma_*^2}\sqrt{\frac{\log^3 d}{n}} \log n.
\]
\end{corollary}
%

%
As we see in the following proposition, if $\sigma_*$ is bounded away from 0 by an absolute constant, $\sqrt{\frac{\log^3d}{n}}$ is generally the optimal convergence rate in this situation, so the above corollary gives a nearly optimal rate.  
%
\begin{proposition}\label{p1}
Let $X=(X_{ij})_{i,j=1}^\infty$ be an array of i.i.d.~random variables such that $Ee^{c|X_{ij}|}<\infty$ for some $c>0$, $EX_{ij}=0$, $EX_{ij}^2=1$ and $\gamma:=EX_{ij}^3\neq0$. 
Let $W=n^{-1/2}\sum_{i=1}^n X_i$ with $X_i:=(X_{i1},\dots,X_{id})^\top$. 
Suppose that $d$ depends on $n$ so that $(\log^3d)/n\to0$ and $d(\log^3d)/n\to\infty$ as $n\to\infty$. 
Also, let $Z\sim N(0, I_d)$.
Then we have
\[
\limsup_{n\to\infty}\sqrt{\frac{n}{\log^3d}}\sup_{x\in\mathbb{R}}\left|P\left(\max_{1\leq j\leq d}W_j\leq x\right)-P\left(\max_{1\leq j\leq d}Z_j\leq x\right)\right|>0.
\]
\end{proposition}  
Proposition~\ref{p1} is proved in Section \ref{sec3-p1}. 
Note that it is possible to find an example which simultaneously satisfies the assumptions in both Corollary~\ref{c0} and Proposition~\ref{p1}. 
In fact, in the setting of Proposition~\ref{p1}, if the law of $X_{ij}$ has a log-concave density,
the assumptions of Corollary~\ref{c0} are satisfied due to the independence across the coordinates of $X_i$. 
For example, this is the case when $X_{ij}$ follows a normalized exponential distribution. 

Theorem~\ref{t1} is also interesting in the context of the so-called Malliavin-Stein method (see \cite{NoPe12} for an exposition of this topic). 
For simplicity, we focus on the case that the coordinates of $W$ are multiple Wiener-It\^{o}integrals with common orders. We refer to \cite{NoPe12} for unexplained concepts appearing in the following corollary (and its proof). 
\begin{corollary}\label{c-wiener}
Let $X$ be an isonormal Gaussian process over a real separable Hilbert space $\mathfrak{H}$. 
Let $q\in\mathbb{N}$ and denote by $I_q(f)$ the $q$-th multiple Wiener-It\^{o}integral of $f\in\mathfrak{H}^{\odot q}$ with respect to $X$, where $\mathfrak{H}^{\odot q}$ denotes the $q$-th symmetric tensor power of $\mathfrak{H}$. 
For every $j=1,\dots,d$, suppose $W_j=I_q(f_j)$ for some $f_j\in\mathfrak{H}^{\odot q}$. Suppose also 
$\Cov(W)=\Sigma$ with diagonal entries all equal to 1 and smallest eigenvalue $\sigma_*^2>0$. Let $Z\sim N(0, \Sigma)$.
Then we have
\begin{equation}\label{eq:c-wiener}
\sup_{h=1_A: A\in \mathcal{R}} |E h(W)- E h(Z)|
\leq C_q \frac{\ol\Delta_W}{\sigma_*^2}(\log^q d)(|\log \ol\Delta_W |\vee1),
\end{equation}
where $C_q>0$ is a constant depending only on $q$ and 
\[
\ol\Delta_W:=\max_{1\leq j\leq d}\sqrt{EW_j^4-3(EW_j^2)^2}.
\]
\end{corollary}
Corollary~\ref{c-wiener} is comparable with Corollary 1.3 in \cite{NoPeYa20}, where an analogous bound to \eqref{eq:c-wiener} is obtained when $\mathcal{R}$ is replaced by the set of all measurable convex subsets of $\mathbb{R}^d$ (see also \cite{KiPa15} for related results). 
Meanwhile, considering the restricted class $\mathcal{R}$, we improve the dimension dependence of the bounds: Typically, the bound of Corollary 1.3 in \cite{NoPeYa20} depends on the dimension through $d^{41/24+1}$, while our bound depends through $\log^qd$.

We also remark that \cite{NoPeYa20} succeeded in removing the logarithmic factor from their bound by an additional recursion argument. However, it does not seem straightforward to apply their argument to our situation.

Stein kernels do not exist for discrete random vectors. Next, we apply other commonly used approaches in Stein's method to exploit the dependence structure of a random vector and obtain error bounds in the normal approximation. First, we consider the exchangeable pair approach developed in \cite{St86} in one-dimensional normal approximations and \cite{ChMe08} and \cite{ReRo09} for multivariate normal approximations. 

\begin{theorem}[Error bound using exchangeable pairs]\label{t2}
Suppose we can construct another random vector $W'$ on the same probability space such that $(W, W')$ and $(W', W)$ have the same distribution (exchangeable), and moreover,
\begin{equation}\label{eq:lr}
E(W'-W\mid W)=-\Lambda (W+R)
\end{equation}
for some invertible $d\times d$ matrix $\Lambda$ (linearity condition). Let $D=W'-W$ and suppose $E\|D\|_\infty^4<\infty$. 
Also, let $Z\sim N(0,\Sigma)$.
Then, if $\eta\geq0$ and $t>0$ satisfy $\eta/\sqrt{t}\leq \sigma_*/\sqrt{\log d}$, we have
\begin{multline*}
\sup_{h=1_A: A\in \mathcal{R}} |E h(W)- E h(Z)|\\
\leq C\left\{\frac{1}{\sigma_*}E\left(\max_{1\leq j\leq d} |R_j|\right)\sqrt{\log d}  
+\frac{1}{\sigma_*^2}\Delta_1 (|\log t|\vee1)\log d
+\frac{1}{\sigma_*^4}(\Delta_2+\Delta_3(\eta))\frac{(\log d)^{2}}{t}  
+\frac{\ol \sigma}{\ul \sigma}\sqrt{t} \log d\right\},
\end{multline*}
where the $\sigma$'s are defined in \eq{eq:sigma},
\ba{
\Delta_1&:=E\left[ \max_{1\leq j,k\leq d}\left|\Sigma_{jk}-\frac{1}{2} E[(\Lambda^{-1}D)_j D_k\mid W]\right| \right],\\
\Delta_2&:=E\left[\max_{1\leq j,k,l,m\leq d}E[|(\Lambda^{-1}D)_jD_kD_lD_m|\mid W]\right],\\
\Delta_3(\eta)&:=E\left[\max_{1\leq j,k,l,m\leq d}|(\Lambda^{-1}D)_jD_kD_lD_m|;\|D\|_{\infty}>\eta\right].
}
\end{theorem}

\begin{remark}
The exchangeability and the linearity condition in the statement of Theorem \ref{t2} may be motivated by considering a bivariate normal vector $(W,W')$ with correlation $\rho$ and $E(W'-W|W)=-(1-\rho) W$.
If $W=\sum_{i=1}^n \xi_i$ is a sums of independent random vectors and $W'=W-\xi_I+\xi_I'$, where $I$ is an independent random index uniformly chosen from $\{1,\dots, n\}$ and $\{\xi_i': 1\leq i\leq n\}$ is an independent copy of $\{\xi_i: 1\leq i\leq n\}$, then it can be verified that \eq{eq:lr} is satisfied with $\Lambda=\frac{1}{n} I_d$ and $R=0$.
The exchangeable pair approach was proved to be useful for dependent random vectors as well; See \cite{ReRo09} and the references therein for many applications.
\end{remark}

\begin{remark}
We can take $\eta=0$ and $t=(\frac{\ul \sigma}{\ol \sigma \sigma_*^4}\Delta_3(0) \log d)^{2/3}$ in Theorem~\ref{t2} to obtain a simpler bound
\be{
C\left\{\frac{1}{\sigma_*}E\left(\max_{1\leq j\leq d} |R_j|\right)\sqrt{\log d}+\frac{1}{\sigma_*^2} \Delta_1 (|\log (\frac{\ul \sigma}{\ol \sigma \sigma_*^4}\Delta_3(0))|\vee 1)\log d   +(\frac{\ol \sigma^2}{\ul \sigma^2 \sigma_*^4}\Delta_3(0) \log^4 d)^{1/3}\right\}.
}
We can simplify the bound in Theorem~\ref{ft4} below similarly. 
However, these simplified bounds result in a worse bound $C(B_n^{4} (\log^{6} d)/\sigma_*^4 n)^{1/3}$ for Corollary~\ref{c2}.
We introduce the parameter $\eta$ in the same spirit as in the Chernozhukov-Chetverikov-Kato theory: It plays a similar role to the parameter $\gamma$ in \cite{CCK13} and serves for better control of maximal moments appearing in the bound. 
\end{remark}

We note that \cite{Me06} introduced an infinitesimal version of the exchangeable pairs approach. Her method can be used to find the Stein kernel for certain random vectors with a continuous symmetry; hence, we can apply Theorem~\ref{t1} to obtain a near optimal rate of convergence. In general, however, the convergence rate obtained using Theorem~\ref{t2} is slower, as demonstrated below in Corollary~\ref{c2}.

Next, we consider non-linear statistics along the lines of \cite{Ch08a}, \cite{ChRo10} and \cite{Du19}.

%

\begin{theorem}[Error bound for non-linear statistics]\label{ft4}
Let $X=(X_1,\dots, X_n)$ be a sequence of independent random variables taking values in a measurable space $\mathcal{X}$.
Let $F: \mathcal{X}^n\to \mathbb{R}^d$ be a measurable function, and let $W=F(X)$. 
Let $X'=(X_1',\dots, X_n')$ be an independent copy of $X$.
For each $A\subset [n]$, define $X^A=(X_1^A,\dots, X_n^A)$ where 
\be{
X_i^A=
\begin{cases}
X_i', & \text{if}\ i\in A\\
X_i, & \text{if}\ i\notin A.
\end{cases}
}
Let $W^A=F(X^A)$.
Suppose $E(W)=0$ and $E\|W\|_\infty^4<\infty$. 
Also, let $Z\sim N(0, \Sigma)$.
Then, if $\eta\geq0$ and $t>0$ satisfy $\eta/\sqrt{t}\leq \sigma_*/\sqrt{\log d}$, we have
\bes{
\sup_{h=1_A: A\in \mathcal{R}} |E h(W)- E h(Z)|\leq C\left(\frac{1}{\sigma_*^2} \delta_1 (|\log t|\vee1) \log d+\frac{1}{\sigma_*^4}(\delta_2+\delta_3(\eta))\frac{1}{t} (\log d)^2+\frac{\ol \sigma}{\ul \sigma}\sqrt{t} \log d\right),
}
where the $\sigma$'s are defined in \eq{eq:sigma},
\ba{
\delta_1&:= E\left[ \max_{1\leq j,k\leq d} \left|\Sigma_{jk}-\frac{1}{2}\sum_{i=1}^n (W^{\{1:i\}}-W^{\{1:(i-1)\}} )_j (W^{\{i\}}-W)_k    \right|    \right],\\
\delta_2&:= E\left[\max_{1\leq j\leq d}\sum_{i=1}^n (  W^{\{1:i\}}-W^{\{1:(i-1)\}} )_j^4  \right]
+E\left[\max_{1\leq j\leq d}\sum_{i=1}^n (  W^{\{i\}}-W )_j^4  \right],\\
\delta_3(\eta)&:=\sum_{i=1}^nE\left[\max_{1\leq j\leq d}(  W^{\{1:i\}}-W^{\{1:(i-1)\}} )_j^4; \|W^{\{i\}}-W\|_{\infty}>\eta  \right]\\
&\quad+\sum_{i=1}^n E\left[\max_{1\leq j\leq d}(  W^{\{i\}}-W )_j^4;\|W^{\{i\}}-W\|_\infty>\eta  \right],
}
and 
\be{
\{1:i\}:=
\begin{cases}
\{1,2,\dots, i\}, &\text{if}\ i\geq 1\\
\emptyset, & \text{if}\ i=0.
\end{cases}
}

\end{theorem}


Using either Theorem~\ref{t2} or \ref{ft4} with a truncation argument, we can improve \cite[Theorem 2.1]{CCKK19} in the strongly non-singular case. 

\begin{corollary}\label{c2}
Let $W=n^{-1/2}\sum_{i=1}^n X_i\in \mathbb{R}^d$, where $\{X_1,\dots, X_n\}$ are centered independent variables with $\Cov(W)=\Sigma$ with diagonal entries all equal to 1 and smallest eigenvalue $\sigma_*^2>0$. Let $Z\sim N(0, \Sigma)$. Suppose that there is a constant $B_n\geq1$ such that $\max_{i,j}E\exp(X_{ij}^2/B_n^2)\leq2$ and $\max_j n^{-1}\sum_{i=1}^nEX_{ij}^4\leq B_n^2$,
where $X_{ij}$ denotes the $j$th component of the vector $X_i$.
Then we have
\begin{align}\label{f15}
\sup_{h=1_A: A\in \mathcal{R}} |E h(W)- E h(Z)|
\leq C\left(\frac{B_n^2\log^4(dn)}{\sigma_*^4 n}\right)^{1/3}.
\end{align}
\end{corollary}


Finally, we consider sums of random vectors with a local dependence structure. 
Unlike in Theorems~\ref{t2} and \ref{ft4}, there is no longer an underlying symmetry that we can exploit.
In the end, we obtain a third-moment error bound with a slower convergence rate.

\begin{theorem}[Error bound for sums of locally dependent variables]\label{ft5}
Let $W=\sum_{i=1}^n X_i$ with $E X_i=0$ and $\Cov(W)=\Sigma$. 
Let $Z\sim N(0,\Sigma)$.
Assume that for each $i\in \{1,\dots, n\}$, there exists $A_i\subset \{1,\dots, n\}$ such that $X_i$ is independent of $\{X_{i'}: i'\notin A_i\}$.
Moreover, assume that for each $i\in \{1,\dots, n\}$ and $i'\in A_i$, there exists $A_{ii'}\subset \{1,\dots, n\}$ such that $\{X_i, X_{i'}\}$ is independent of $\{X_{i''}: i''\notin A_i\}$. Then we have
\besn{\label{f16}
&\sup_{h=1_A: A\in \mathcal{R}} |E h(W)- E h(Z)|\\
\leq & C  \sqrt{\frac{\ol \sigma}{\ul \sigma \sigma_*^3}\sum_{i=1}^n\sum_{i'\in A_i}\sum_{i''\in A_{ii'}}E\left[ \max_{1\leq j,k,l\leq d} (|X_{ij}X_{i'k}X_{i''l}| +|X_{ij}X_{i'k}|E|X_{i''l}|) \right]  (\log d)^{5/2}},
}
where the $\sigma$'s are defined in \eq{eq:sigma}.
\end{theorem}
Theorem~\ref{ft5} may be improved using a truncation as in Theorems~\ref{t2} and \ref{ft4}. We leave it as it is for simplicity.

\subsection{Literature on multivariate normal approximations}\label{sec1.1}
There is a large literate on multivariate normal approximations.
Here we discuss some of the recent results providing error bounds on various distributional distances with the best-known dependence on dimension. 

Let $W=n^{-1/2}\sum_{i=1}^n X_i\in \mathbb{R}^d$, where $\{X_1,\dots, X_n\}$ are centered independent variables with $\Cov(W)=\Sigma$.
\cite{Be05} proved that, with $Z\sim N(0,\Sigma)$,
\ben{\label{r01}
\sup_{h=1_A: A\in \mathcal{C}} |E h(W)- E h(Z)|\leq  \frac{C d^{1/4}}{n^{3/2}} \sum_{i=1}^n E[|\Sigma^{-1/2} X_i|^3],
}
where $\mathcal{C}$ is the collection of all (measurable) convex sets of $\mathbb{R}^d$ and $|\cdot|$ denotes the Euclidean norm when applied to a vector.
The bound \eq{r01} is optimal up to the $d^{1/4}$ factor (\cite{Na76}).
See \cite{Ra19} for a bound with explicit constant.
In the typical case where $E[|\Sigma^{-1/2}X_i|^3]$ is of the order $O(d^{3/2})$, the error bound in \eq{r01} is of the order $O(\frac{d^{7/2}}{n})^{1/2}$. \cite{CCK13} and subsequent works showed that by restricting to the class of hyperrectangles, one may obtain much better dependence on $d$.

If we restrict to the class of Euclidean balls and assume $\Sigma=I_d$, then we can obtain a bound as in \eq{r01} but without the factor $d^{1/4}$. This follows from \cite[Theorem 1.3 and Example 1.2]{Ra19} and \cite[Remark 2.1]{Sa72}, for example.
For Euclidean balls centered at 0, it is known that (cf. \cite{GoZa14}) the dependence on $n$ may be improved from $1/\sqrt{n}$ to $1/n$ for $d\geq 5$, which is in general the smallest possible dimension for such an improvement. See also \cite{Es45}, \cite{BeGo97}, \cite{GoUl03}, \cite{BoGoUl06} and \cite{PrUl13} for earlier and related results.
We do not know any results with explicit dependence on $d$ and such improved dependence on $n$.

Another class of distributional distances of interest is the $L^p$ transportation distance for a number $p\geq 1$, also known as the Kantorovich distance or the $p$-Wasserstein distance. 
For two probability measures $\mu$ and $\nu$ on $\mathbb{R}^d$, it is defined to be
\be{
\mathcal{W}_p(\mu, \nu):=\left( \inf_{\gamma\in \Lambda(\mu,\nu)} \int |x-y|^p d\gamma(x,y)  \right)^{1/p},
}
where $\Lambda(\mu,\nu)$ is the space of all probability measures on $\mathbb{R}^d\times \mathbb{R}^d$ with $\mu$ and $\nu$ as marginals.
If $X$ and $Y$ are random variables with distributions $\mu$ and $\nu$, respectively, we will also write
\be{
\mathcal{W}_p(X, Y)=\mathcal{W}_p(\mu,\nu).
}

Let $W=n^{-1/2}\sum_{i=1}^n X_i\in \mathbb{R}^d$, where $\{X_1,\dots, X_n\}$ are centered i.i.d.\ variables with $\Cov(W)=\Sigma$, and let $Z\sim N(0,\Sigma)$. Suppose $|X_i|\leq \beta$ almost surely for some $\beta>0$. \cite{EMZ18} proved that
\ben{\label{r02}
\mathcal{W}_2(W, Z)\leq \frac{\beta \sqrt{d} \sqrt{32+2\log_2(n)}}{\sqrt{n}}.
}
The bound in \eq{r02} is optimal up to the $\log_2(n)$ factor (\cite{Zh18}). See \cite{CFP19} and \cite{EMZ18} for results for the log-concave case.
Following the proof of Proposition 1.4 of \cite{Zh18}, these bounds on the $L^2$ transportation distance can be used to deduce a bound on $\sup_{h=1_A: A\in \mathcal{R}} |E h(W)- E h(Z)|$. 
For example, we can obtain the following proposition. We defer its proof to the end of the Appendix.

\begin{proposition}\label{p2}
Let $T$ be any $\mathbb{R}^d$-valued random variable. 
Let $Z\sim N(0,\Sigma)$ where $\Sigma_{jj}\geq 1, \forall\ 1\leq j\leq d$.
Then,
\be{
\sup_{h=1_A: A\in \mathcal{R}} |E h(T)- E h(Z)|\leq C(\log d)^{1/3} \mathcal{W}_2(T, Z)^{2/3}.
}
\end{proposition}

Applying Proposition \ref{p2} to \eq{r02}, we have the following corollary. 

\begin{corollary}\label{c3}
Let $W=n^{-1/2}\sum_{i=1}^n X_i\in \mathbb{R}^d$, where $\{X_1,\dots, X_n\}$ are centered i.i.d.\ variables with $\Cov(W)=\Sigma$. Suppose $\Sigma_{jj}\geq 1, \forall\ 1\leq j\leq d$. Suppose further that $|X_i|\leq \beta$ almost surely for some $\beta>0$.
Let $Z\sim N(0,\Sigma)$.
Then,
\be{
\sup_{h=1_A: A\in \mathcal{R}} |E h(W)- E h(Z)|\leq C (\log d)^{1/3} \frac{d^{1/3} \beta^{2/3}}{n^{1/3}}(1+\log n)^{1/3}.
}
\end{corollary}

Since the $\mathcal{W}_2$ error bound in \eq{r02} scales like $\sqrt{d}$, we see that such a deduced bound again requires $d$ to be sub-linear in $n$. Allowing to go well beyond this restriction is a key feature of this paper.

If we assume in addition that the mixed third moments of $X_1$ are all equal to zero, then it is possible to improve the dependence on $n$ from $1/\sqrt{n}$ to $1/n$. See, for example, \cite{BoChGo13} for such improved rate in total variation in dimension one and \cite{Fa18} for results on the 2-Wasserstein distance in multi-dimensions.

\section{Proofs}\label{sec2}

\subsection{Lemmas}

We first state four lemmas that are needed in the proofs of the main results. 
Set $R(0; \epsilon):=\{x\in\mathbb R^d:\|x\|_\infty\leq\epsilon\}$ for $\epsilon>0$. 
Throughout this section, we denote by $\phi$ the density function of the standard $d$-dimensional normal distribution.

\begin{lemma}[Gaussian anti-concentration inequality]\label{l1}
Let $Y$ be a centered Gaussian vector in $\mathbb{R}^d$ such that $\min_{1\leq j\leq d}EY_j^2\geq\underline{\sigma}^2$ for some $\underline{\sigma}>0$. Then, for any $y\in\mathbb{R}^d$ and $\varepsilon>0$,
\[
P(Y\leq y+\varepsilon)-P(Y\leq y)\leq\frac{\varepsilon}{\underline{\sigma}}\left(\sqrt{2\log d}+2\right),
\]
where $\{Y\leq y\}:=\{Y_j\leq y_j: 1\leq j\leq d\}$.
\end{lemma}
A proof of Lemma \ref{l1} is found in \cite{CCK17nazarov}.

 
\begin{lemma}[Modification of (2.10) of \cite{AnHaTi98}]\label{L2}
Let $K\geq0$ be a constant and set $\eta=\eta_d:=K/\sqrt{\log d}$. 
Then, for all $r\in\mathbb{N}$, there is a constant $C_{K,r}>0$ depending only on $K$ and $r$ such that 
\begin{equation*}
\sup_{A\in\mathcal{R}}\sum_{j_1,\dots,j_r=1}^d\sup_{y\in R(0;\eta)}\left|\int_A\partial_{j_1,\dots, j_r}\phi(z+y)dz\right|
\leq C_{K,r} (\log d)^{r/2}.
\end{equation*}
\end{lemma}
The special version of Lemma~\ref{L2} with $\eta=0$ is found in the proof of (2.10) of \cite{AnHaTi98}. 
Introduction of the parameter $\eta$ is motivated by a standard argument used in the Chernozhukov-Chetverikov-Kato theory to efficiently control maximal moments appearing in normal approximation error bounds; see Equation (24) in \cite{CCK13}, for example. 

To clarify the structure of the proof, we first give a proof of the case with $r=1$ and $\eta=0$ here. 
The proof of the general case follows the same strategy and will be given in Section \ref{sec3-l2}, where we need a few technical lemmas and more complicated notation.
\begin{proof}[Proof of Lemma \ref{L2} with $r=1$ and $\eta=0$]
We denote by $\phi_1$ and $\Phi_1$ the density and distribution function of the standard normal distribution, respectively. We set $\bar\phi_1(u):=\phi_1(u)/\Phi_1(u)$. 

For any $A=\prod_{j=1}^d(a_j,b_j)\in\mathcal{R}$, we have, by considering $x_j:=|a_j|\wedge|b_j|$ in the first inequality, 
\besn{\label{AHT-est}
\sum_{j=1}^d\left|\int_A\partial_{j}\phi(z)dz\right|
&=\sum_{j=1}^d\left|\phi_1(b_j)-\phi_1(a_j)\right|\prod_{k:k\neq j}\left(\Phi_1(b_k)-\Phi_1(a_k)\right)\\
&\leq\sup_{x\in\mathbb{R}^d}\sum_{j=1}^d\phi_1(x_j)\prod_{k:k\neq j}\Phi_1(x_k)
=\sup_{x\in\mathbb{R}^d}\sum_{j=1}^d\bar\phi_1(x_j)\prod_{k=1}^d\Phi_1(x_k).
}
Therefore, it suffices to prove $\sup_{x\in\mathbb{R}^d}f(x)=O(\sqrt{\log d})$, where
\[
f(x)=F(x)G(x)\quad\text{with}\quad
F(x)=\sum_{j=1}^d\bar\phi_1(x_j)
\quad\text{and}\quad
G(x)=\prod_{k=1}^d\Phi_1(x_k).
\]

The remaining proof proceeds as follows. We first show that $f$ has a maximizer $x^*$ satisfying $x^*_1=\cdots=x^*_d=:u^*$. From this, we will see $\sup_{x\in\mathbb{R}^d}f(x)=O(u^*)$ and $u^*=O(\sqrt{\log d})$ as $d\to\infty$. This completes the proof. 

From the last identity in \eqref{AHT-est} we have $\sup_{x\in\mathbb{R}^d}f(x)=\sup_{x\in[0,\infty)^d}f(x)$. Also, noting that $\bar\phi_1'(u)=-(u+\bar\phi_1(u))\bar\phi_1(u)$, we have
\begin{align*}
\partial_lf(x)&=\bar\phi_1'(x_l)G(x)+F(x)\bar\phi_1(x_l)G(x)\\
&=\{-(x_l+\bar\phi_1(x_l))+F(x)\}\bar\phi_1(x_l)G(x).
\end{align*}  
In particular, $\partial_lf(x)<0$ if $x\in[0,\infty)^d$ and $x_l>d$ because $F(x)\leq d\sqrt{2/\pi}\leq d$ for all $x\in[0,\infty)^d$. This means $f(x)\leq f(x_1\wedge d,\dots,x_d\wedge d)$ for all $x\in[0,\infty)^d$, and thus we obtain $\sup_{x\in\mathbb{R}^d}f(x)=\sup_{x\in[0,d]^d}f(x)$. 
As a result, $f$ has a maximizer $x^*$ satisfying $x^*\in[0,d]^d$ and $\nabla f(x^*)=0$. The latter equation yields
\begin{equation}\label{1st-order}
x^*_1+\bar\phi_1(x^*_1)
=\cdots
=x^*_d+\bar\phi_1(x^*_d)
=F(x^*).
\end{equation}
Now, it is easy to see that the function $[0,\infty)\ni u\mapsto u+\bar\phi_1(u)\in(0,\infty)$ is strictly increasing (this is in fact a special case of Lemma \ref{l2-2}). 
Consequently, we obtain $x_1^*=\cdots=x_d^*=:u^*$. 

Now, from \eqref{1st-order} we have $u^*=(d-1)\bar\phi_1(u^*)$. So we obtain
\[
\sup_{x\in\mathbb{R}^d}f(x)=f(x^*)=\frac{d}{d-1}u^*\Phi_1(u^*)^d
\leq\frac{d}{d-1}u^*.
\] 
Therefore, we complete the proof once we prove $u^*=O(\sqrt{\log d})$. 

Setting $g_2(u):=u-(d-1)\bar\phi_{1}(u)$, we have $g_2(u^*)=0$ and $g_2(\sqrt{2\log d})\to\infty$ as $d\to\infty$. Since $g_2$ is increasing on $[0,\infty)$, we conclude $u^*=O(\sqrt{\log d})$ as $d\to\infty$. 
\end{proof}

From Lemma \ref{L2}, we can obtain the following lemma.
For any bounded measurable function $f:\mathbb{R}^d\to\mathbb{R}$ and $\sigma>0$, we define the function $\mcl{N}_\sigma f:\mathbb{R}^d\to\mathbb{R}$ by
\ben{\label{def:n}
\mcl{N}_\sigma f(x)=\int_{\mathbb{R}^d}f(x+\sigma z)\phi(z)dz
=\int_{\mathbb{R}^d}f(\sigma z)\phi(z-x/\sigma)dz.
}
Note that $\mcl{N}_\sigma f$ is infinitely differentiable.  
\begin{lemma}\label{l2}
Let $K\geq0$ be a constant and set $\eta=\eta_d:=K/\sqrt{\log d}$. 
Then, for all $r\in\mathbb{N}$, there is a constant $C_{K,r}>0$ depending only on $K$ and $r$ such that 
\begin{equation*}
\sup_{h=1_A: A\in\mathcal{R}}\sup_{x\in\mathbb{R}^d}\sum_{j_1,\dots,j_r=1}^d\sup_{y\in R(0;\sigma\eta)}\left|\partial_{j_1,\dots, j_r}\mcl{N}_\sigma h(x+y)\right|
\leq C_{K,r} \sigma^{-r}(\log d)^{r/2}
\end{equation*}
for any $h=1_A$, $A\in\mathcal{R}$ and $\sigma>0$.
\end{lemma}
\begin{proof}
Fix $h=1_A$, $A\in\mathcal{R}$, arbitrarily. For any $x,y\in\mathbb{R}^d$ and $j_1,\dots,j_r\in\{1,\dots,d\}$, we have
\ba{
\partial_{j_1,\dots, j_r}\mcl{N}_\sigma h(x+y)
&=(-1)^r\sigma^{-r}\int_{\mathbb{R}^d}h(\sigma z)\partial_{j_1,\dots, j_r}\phi(z-(x+y)/\sigma)dz\\
&=(-1)^r\sigma^{-r}\int_{\mathbb{R}^d}h(x+\sigma z)\partial_{j_1,\dots, j_r}\phi(z-y/\sigma)dz\\
&=(-1)^r\sigma^{-r}\int_{\sigma^{-1}(A-x)}\partial_{j_1,\dots, j_r}\phi(z-y/\sigma)dz,
} 
where $\sigma^{-1}(A-x):=\{\sigma^{-1}(z-x):z\in A\}\in\mcl{R}$. Hence we obtain
\ba{
\sum_{j_1,\dots,j_r=1}^d\sup_{y\in R(0;\sigma\eta)}\left|\partial_{j_1,\dots, j_r}\mcl{N}_\sigma h(x+y)\right|
\leq\sigma^{-r}\sup_{A\in\mathcal{R}}\sum_{j_1,\dots,j_r=1}^d\sup_{y\in R(0;\eta)}\left|\int_{A}\partial_{j_1,\dots, j_r}\phi(z+y)dz\right|.
}
Now, the desired result follows from Lemma \ref{L2}. 
\end{proof}

Next, we state a smoothing lemma. The test function $h=1_A$ we deal with in bounding 
$\sup_{h=1_A: A\in \mathcal{R}} |E h(W)- E h(Z)|$ is not continuous. It is a common strategy to smooth it first, then quantify the error introduced by such smoothing, finally balance the smoothing error with the smooth test function bound. We follow \cite{BhRa76} to smooth $h$ by convoluting it with a Gaussian distribution $K$ with a small variance.

\begin{lemma}[Modification of Lemma 11.4 of \cite{BhRa76}]\label{l3}
Let $\mu, \nu, K$ be probability measures on $\mathbb{R}^d$.
Let $\epsilon>0$ be a constant such that 
\be{
\alpha:=K(R(0;\epsilon))>1/2.
}
Let $h: \mathbb{R}^d\to \mathbb{R}$ be a bounded measurable function. 
Then we have
\be{
\left|\int h d(\mu-\nu)\right|\leq (2\alpha-1)^{-1}[\gamma^*(h;\epsilon)+\tau^*(h; 2\epsilon)],
}
where
\be{
\gamma^*(h;\epsilon)=\sup_{y\in \mathbb{R}^d}\gamma(h_y;\epsilon),\qquad
\tau^*(h;2\epsilon)=\sup_{y\in \mathbb{R}^d}\tau(h_y;2\epsilon),\qquad
h_y(x)=h(x+y),
}
\be{
\gamma(h;\epsilon)=\max\left\{\int M_h(\cdot; \epsilon) d(\mu-\nu)*K, -\int m_h(\cdot; \epsilon)  d(\mu-\nu)*K   \right\},
}
\be{
\tau(\cdot;2\epsilon)= \int [M_h(\cdot; 2\epsilon)-m_h(\cdot; 2\epsilon)] d\nu,
}
\be{
M_h(x;\epsilon)=\sup_{y: \norm{y-x}_\infty\leq \epsilon} h(y),\qquad m_h(x;\epsilon)=\inf_{y: \norm{y-x}_\infty\leq \epsilon} h(y),
}
and $*$ denotes the convolution of two probability measures.
\end{lemma}
Lemma \ref{l3} can be shown in a completely parallel way to that of Lemma 11.4 in \cite{BhRa76} by changing the $\epsilon$-balls therein to $\epsilon$-rectangles, so we omit its proof.

\subsection{Basic estimates} 


In Theorems~\ref{t1}--\ref{ft5}, we aim to bound
\be{
\delta:=\sup_{h=1_A: A\in \mathcal{R}} |E h(W)- E h(Z)|,\quad Z\sim N(0,\Sigma).
}
In this subsection, we collect some basic estimates used in all of their proofs.
Fix $A\in \mathcal{R}$. Let
\be{
h=1_A,\quad \tilde{h}=1_A-P(Z\in A).
}
For $s>0$, let
\be{
T_s\tilde h(x)=\E\tilde h (e^{-s} x+\sqrt{1-e^{-2s}}Z).
}
Note that $E T_s \tilde h(Z)=E \tilde h(Z)=0.$
Let $Q$ and $G$ be the laws of $W$ and $Z$, respectively. 
For a probability distribution $\mu$ on $\mathbb{R}^d$ and $\sigma\geq0$, we denote by $\mu_\sigma$ the law of the random vector $\sigma Y$ with $Y\sim\mu$.
For $t>0$ to be chosen, we are going to apply Lemma~\ref{l3} with 
\be{
\mu=Q_{e^{-t}},\ \nu=G_{e^{-t}}, \  K=G_{\sqrt{1-e^{-2t}}}, \   h=1_A,
}
and $\epsilon$ be such that
\be{
G_{\sqrt{1-e^{-2t}}}\{\norm{z}_\infty\leq \epsilon\}=7/8.
}
We first bound $\tau^*(h; 2\epsilon)$ in Lemma~\ref{l3}.
Recall the definition of $\sigma$'s from \eq{eq:sigma}.
Markov's inequality and Lemma 2.1 of \cite{Ch08b} yield
\be{
\epsilon\leq C\sqrt{1-e^{-2t}}\E\|Z\|_\infty \leq C\ol{\sigma}\sqrt{1-e^{-2t}}\sqrt{\log d}.
}
Thus, applying the Gaussian anti-concentration inequality in Lemma \ref{l1} with $Y=(e^{-t}Z^\top,-e^{-t}Z^\top)^\top$, we obtain
\ben{\label{2}
\tau^*(h; 2\epsilon)\leq Ce^t \frac{\log d}{\ul \sigma}\epsilon \leq \frac{C\ol \sigma}{\ul \sigma} e^t \sqrt{t} \log d,
}
where we used the elementary inequality $1-e^{-x}\leq x$ for all $x\geq0$.

Now we turn to bounding $\gamma^*(h;\epsilon)$ in Lemma~\ref{l3}. Note that $M_h(\cdot;\epsilon)$ and $m_h(\cdot; \epsilon)$ are again indicator functions of rectangles. 
Note also that the class $\mathcal{R}$ is invariant under translation and scalar multiplication. 
Therefore, it suffices to obtain a uniform upper bound for the absolute value of
\ben{\label{1}
\int h d(\mu-\nu)*K=\int \tilde h d\mu*K=ET_t \tilde h (W)
}
over all $h=1_A, A\in \mathcal{R}$. 
In fact, we have by Lemma~\ref{l3} and \eqref{2}--\eqref{1}
\ben{\label{3}
\delta\leq C\left(
\sup_{h=1_A: A\in \mathcal{R}}|ET_t \tilde h (W)|
+\frac{\ol \sigma}{\ul \sigma}e^t \sqrt{t} \log d
\right).
}

We use (various versions of) Stein's method to bound $ET_t \tilde h (W)$. 
Similarly to (1.14) and (3.1) of \cite{BhHo10}, one can verify that 
\be{
\psi_t(x)=-\int_t^\infty T_s\tilde h(x)ds
}
is a solution to the Stein equation
\be{
\langle \Sigma,\text{Hess}\psi_t(w)\rangle_{H.S.}-w \cdot \nabla \psi_t(w)=T_t \tilde h(w).
}
Thus we have
\begin{equation}\label{eq:stein}
\E T_t\tilde h(W)=\E[\langle \Sigma, \text{Hess}\psi_t(W)\rangle_{H.S.}-W\cdot \nabla \psi_t(W)].
\end{equation}
Set $\tilde\Sigma:=\Sigma-\sigma_*^2I_d$. Note that $\tilde\Sigma$ is positive semidefinite because $\sigma_*^2$ is the smallest eigenvalue of $\Sigma$. Let us take independent random vectors $\tilde Z$ and $Z'$ such that $\tilde Z\sim N(0,\tilde\Sigma)$, $Z'\sim N(0,I_d)$ and they are independent of everything else. 
Then, since $\tilde Z+\sigma_* Z'\sim N(0,\Sigma)$, we can rewrite $T_s\tilde h(x)$ as
\besn{\label{eq:stein-sol}
T_s\tilde h(x)&=\E\tilde h(e^{-s}x+\sqrt{1-e^{-2s}}\tilde Z+\sigma_*\sqrt{1-e^{-2s}}Z')\\
&=\E\mcl{N}_{\sigma_*\sqrt{1-e^{-2s}}}\tilde h(e^{-s}x+\sqrt{1-e^{-2s}}\tilde Z),
}
where $\mcl{N}_{\sigma_*\sqrt{1-e^{-2s}}}\tilde h$ is defined by \eqref{def:n}. 
Therefore, applying Lemma \ref{l2} with $\eta=0$ and noting (3.14)--(3.15) of \cite{BhHo10}, we obtain
\begin{align}
\sum_{j=1}^d\left|\partial_{j}\psi_t(x)\right|
&\leq C\sigma_*^{-1}\sqrt{\log d},\label{aht1}\\
\sum_{j,k=1}^d\left|\partial_{jk}\psi_t(x)\right|
&\leq C\sigma_*^{-2}(|\log t| \vee 1)(\log d),\label{aht2}\\
\sum_{j,k,l=1}^d\left|\partial_{jkl}\psi_t(x)\right|
&\leq C\sigma_*^{-3}\frac{1}{\sqrt t}(\log d)^{3/2}.\label{aht3}
\end{align}

\subsection{Proof of Theorem~\ref{t1}}


Without loss of generality, we may assume $\frac{\Delta_W}{\sigma_*^2}<1$; otherwise, the bound \eq{f12} is trivial. 
Since $W$ has a Stein kernel $\tau^W$, we obtain by \eqref{eq:stein}
\[
E T_t\tilde h(W)=E\left[\sum_{j,k=1}^d\partial_{jk}\psi_t(W)(\Sigma_{jk}-\tau^W_{jk}(W))\right].
\]
Therefore, we deduce by \eqref{aht2}  
\begin{align*}
\left|E T_t\tilde h(W)\right|
\leq \frac{C}{\sigma_*^2} (\log d)\Delta_W (|\log t|\vee1).
\end{align*}
Consequently, we have by \eq{3} 
\be{
\delta\leq C \left\{ \frac{\ol \sigma}{\ul \sigma} e^t\sqrt{t} \log d + \frac{1}{\sigma_*^2}(\log d)\Delta_W(|\log t|\vee1)\right\}.
}
Setting $\sqrt{t}=\frac{\ul \sigma \Delta_W}{\ol \sigma \sigma_*^2}$ and noting $\frac{\Delta_W}{\sigma_*^2}< 1$, we obtain the desired result.

\subsection{Proof of Corollary~\ref{c0}}

Without loss of generality, we may assume
\begin{equation}\label{wlog0}
\frac{1}{\sigma_*^2}\sqrt{\frac{\log^3d}{n}}\leq1.
\end{equation}
As in the proof of \cite[Theorem 3.3]{Fathi19}, we first prove the result when $\mu$ is compactly supported and its density is bounded away from zero on its support. Then, by Theorem 2.3 and Proposition 3.2 in \cite{Fathi19}, $\mu$ has a Stein kernel $\tau=(\tau_{jk})_{1\leq j,k\leq d}$ such that $\tau(x)$ is positive definite for all $x\in\mathbb{R}^d$ and $\max_{1\leq j\leq d}E[|\tau_{jj}(X_1)|^p]\leq 8^pp^{2p}$ for all $p\geq1$ (here we used the assumption that $\Var(W_j)=1, \forall\ 1\leq j\leq d$). 
Note that we indeed have $\max_{1\leq j,k\leq d}E[|\tau_{jk}(X_1)|^p]\leq 8^pp^{2p}$ for all $p\geq1$ due to positive definiteness. 
In particular, Lemma A.7 in \cite{Ko19a} (with $q=4$) yields 
\ben{\label{r03}
\max_{1\leq j,k\leq d}E\exp\left(\sqrt{\frac{|\tau_{jk}(X_1)|}{C}}\right)\leq2.
}

Now we define the function $\tau^W:\mathbb{R}^d\to\mathbb{R}^{d\times d}$ by 
\[
\tau^W(x)=E\left[\frac{1}{n}\sum_{i=1}^n\tau(X_i)\mid W=x\right],\qquad x\in\mathbb{R}^d. 
\]
It is well known that $\tau^W$ is a Stein kernel for $W$ (cf.~page 271 of \cite{LeNoPe15}). 
Jensen's inequality yields
\ben{\label{r05}
E\left[\max_{1\leq j,k\leq d}\left|\Sigma_{jk}-\tau^W_{jk}(W)\right|\right]
\leq E\left[\max_{1\leq j,k\leq d}\left|\frac{1}{n}\sum_{i=1}^n(\tau_{jk}(X_i)-\Sigma_{jk})\right|\right].
}
We will use Theorem 3.1 and Remark A.1 in \cite{KuCh18} to bound the right-hand side of \eq{r05}. We need the following definitions.
\begin{definition}[Orlicz Norms]\label{d1}
Let $g: [0,\infty)\to [0, \infty)$ be a non-decreasing function with $g(0)=0$. The ``$g$-Orlicz norm'' of a random variable $X$ is given by
\be{
\norm{X}_g:=\inf\{\eta>0: E[g(|X|/\eta)]\leq 1\}.
}
\end{definition}
\begin{definition}[Sub-Weibull Variables]
A random variable $X$ is said to be sub-Weibull of order $\alpha>0$, denotes as sub-Weibull$(\alpha)$, if
\be{
\norm{X}_{\psi_\alpha}<\infty, \ \text{where}\ \psi_\alpha(x)=\exp(x^\alpha)-1\ \text{for}\ x\geq 0.
}
\end{definition}
\begin{definition}[Generalized Bernstein-Orlicz Norm]
Fix $\alpha>0$ and $L\geq 0$. Define the function $\Psi_{\alpha, L}(\cdot)$ based on the inverse function 
\be{
\Psi_{\alpha, L}^{-1}(t):=\sqrt{\log(1+t)}+L (\log(1+t))^{1/\alpha}\ \text{for all}\ t\geq 0.
}
The Generalized Bernstein-Orlicz (GBO) norm of a random variable $X$ is then given by $\norm{X}_{\Psi_{\alpha, L}}$ as in Definition \ref{d1}.
\end{definition}
Applying Theorem 3.1 of \cite{KuCh18} to the sequence of independent mean zero sub-Weibull$(\frac{1}{2})$ random variable (cf. \eq{r03}) $\{\tau_{jk}(X_i)-\Sigma_{jk}: i=1,\dots, n\}$, we have
\be{
\norm{\frac{1}{n}\sum_{i=1}^n (\tau_{jk}(X_i)-\Sigma_{jk}) }_{\Psi_{\frac{1}{2}, L_n}}\leq \frac{C}{\sqrt{n}},
}
for some $L_n=C/\sqrt{n}$.
Combining with Remark A.1 of \cite{KuCh18}, we have,
\besn{\label{r04}
&E\left[\max_{1\leq j,k\leq d}\left|\frac{1}{n}\sum_{i=1}^n(\tau_{jk}(X_i)-\Sigma_{jk})\right|\right]\\
\leq & C \max_{1\leq j,k\leq d} \norm{\frac{1}{n}\sum_{i=1}^n (\tau_{jk}(X_i)-\Sigma_{jk}) }_{\Psi_{\frac{1}{2}, L_n}} \left(\sqrt{\log d}+L_n \log^2 d\right)\\
\leq & C\frac{1}{\sqrt n}\left(\sqrt{\log d}+\frac{1}{\sqrt n}\log^2d\right).
}
From \eq{r05} and \eq{r04}, we have
\ba{
E\left[\max_{1\leq j,k\leq d}\left|\Sigma_{jk}-\tau^W_{jk}(W)\right|\right]
&\leq C\frac{1}{\sqrt n}\left(\sqrt{\log d}+\frac{1}{\sqrt n}\log^2d\right)\\
&=C\sqrt{\frac{\log d}{n}}\left(1+\frac{\log^{3/2}d}{\sqrt n}\right)
\leq C\sqrt{\frac{\log d}{n}},
}
where the last inequality follows from \eqref{wlog0}.
Therefore, an application of Theorem \ref{t1}, together with the fact that $x(|\log x|\vee 1)$ is an increasing function for $x\geq 0$ and the assumption \eq{wlog0}, yields the desired result. 

Next we prove the result when $\mu$ is supported on the whole space $\mathbb R^d$. We take a sequence of convex bodies $F_\ell$ that converge to $\mathbb{R}^d$. 
Define the probability measure $\nu_\ell$ on $\mathbb{R}^d$ by $\nu_\ell(A)=\mu(A\cap F_\ell)/\mu(F_\ell)$ for any Borel set $A\subset\mathbb R^d$ (note that $\mu(F_\ell)\to1$ as $\ell\to\infty$ by construction, so $\mu(F_\ell)>0$ for sufficiently large $\ell$). 
Then, let $\mu_\ell$ be the law of the variable $M_\ell^{-1/2}(Y_\ell-EY_\ell)$, where $Y_\ell$ is a random vector with law $\nu_\ell$ and $M_\ell$ is the diagonal matrix with the diagonal entries equal to those of $\Cov(\mu_\ell)$ (note that $\Cov(Y_\ell)\to \Sigma$ as $\ell\to\infty$ by construction, so $M_\ell^{-1/2}$ exists for sufficiently large $\ell$). 
Note that $M_\ell \to I_d$. 
Also, the density of $\mu_\ell$ is bounded away from zero on its support because $\mu$ is supported on $\mathbb R^d$ and has a continuous density. 
Hence, letting $W_{\ell}=n^{-1/2}\sum_{i=1}^n X^{(\ell)}_i\in \mathbb{R}^d$ with $\{X^{(\ell)}_1,\dots, X^{(\ell)}_n\}$ being i.i.d.~with law $\mu_\ell$ and using the result for the compactly supported case above, we have, for sufficiently large $\ell$,
\[
\sup_{h=1_A: A\in \mathcal{R}} |E h(W_\ell)- E h(Z)|
\leq \frac{C}{\sigma_*^2}\sqrt{\frac{\log^3 d}{n}}\log n.
\]
Moreover, it is also easy to verify that the density of $W_\ell$ converges almost everywhere to that of $W$ as $\ell\to\infty$. Thus, Scheff\'e's lemma yields
\[
\sup_{h=1_A: A\in \mathcal{R}} |E h(W_\ell)- E h(W)|\to0\qquad(\ell\to\infty).
\]
This yields the desired result.

Finally, to prove the result in the general case, take $\epsilon>0$ arbitrarily and let $\mu_\epsilon$ be the law of the variable $\sqrt{1-\epsilon^2}X_1+\epsilon\zeta$, where $\zeta\sim N(0,\Sigma)$ and is independent of $\{X_1,\dots,X_n\}$. It is evident that $\mu_\epsilon$ has covariance matrix $\Sigma$ and supported on the whole space $\mathbb R^d$. Moreover, $\mu_\epsilon$ has a log-concave density by Proposition 3.5 in \cite{SaWe14}. Hence we have for any $A\in\mathcal{R}$
\[
|E1_A(W)-E1_A(Z)|
\leq|E1_A(W)-E1_A(W^\epsilon)|
+\frac{C}{\sigma_*^2}\sqrt{\frac{\log^3 d}{n}}\log n,
\]
where $W^\epsilon:=\sqrt{1-\epsilon^2}W+\epsilon \zeta$. Since $W$ has a density and $W^\epsilon$ converges in law to $W$ as $\epsilon\to0$, $|E1_A(W)-E1_A(W^\epsilon)|\to0$ as $\epsilon\to0$. Thus, letting $\epsilon\to0$ and taking the supremum over $A\in\mathcal{R}$ in the above inequality, we complete the proof. 

%


%
%

\subsection{Proof of Corollary~\ref{c-wiener}}

By Proposition 3.7 in \cite{NoPeSw14}, $W$ has a Stein kernel $\tau^W=(\tau^W_{jk})_{1\leq j,k\leq d}$ given by
\[
\tau^W_{jk}(x)=E[\langle-DL^{-1}W_j,DW_k\rangle_{\mathfrak{H}}\mid W=x],\qquad x\in\mathbb{R}^d,
\]
where $\langle\cdot,\cdot\rangle_{\mathfrak{H}}$ denotes the inner product of $\mathfrak{H}$, while $D$ and $L^{-1}$ denote the Malliavin derivative and pseudo inverse of the Ornstein-Uhlenbeck operator with respect to $X$, respectively. 
By Jensen's inequality and Lemma 2.2 in \cite{Ko19a}, we have
\ba{
E\left[\max_{1\leq j,k\leq d}\left|\Sigma_{jk}-\tau^W_{jk}(W)\right|\right]
&\leq E\left[\max_{1\leq j,k\leq d}\left|\Sigma_{jk}-\langle-DL^{-1}W_j,DW_k\rangle_{\mathfrak{H}}\right|\right]
\leq C_q(\log^{q-1} d)\ol\Delta_W,
}
where $C_q>0$ depends only on $q$. Thus the desired result follows from Theorem \ref{t1}.

\subsection{Proof of Theorem~\ref{t2}}
Without loss of generality, we may assume $t<1$; otherwise, the theorem is trivial because $\sup_{h=1_A: A\in \mathcal{R}} |E h(W)- E h(Z)|\leq 1$.
By exchangeability we have
\begin{align*}
0&=\frac{1}{2}E[\Lambda^{-1}D\cdot(\nabla \psi_t(W')+\nabla \psi_t(W))]\\
&=E\left[\frac{1}{2}\Lambda^{-1}D\cdot(\nabla \psi_t(W')-\nabla \psi_t(W))+\Lambda^{-1}D\cdot\nabla \psi_t(W)\right]\\
&=E\left[\frac{1}{2}\sum_{j,k=1}^d(\Lambda^{-1}D)_jD_k\partial_{jk}\psi_t(W)+\Xi+\Lambda^{-1}D\cdot\nabla \psi_t(W)\right],
\end{align*}
where
\[
\Xi=\frac{1}{2}\sum_{j,k,l=1}^d(\Lambda^{-1}D)_jD_kD_lU\partial_{jkl}\psi_t(W+(1-U)D)
\]
and $U$ is a uniform random variable on $[0,1]$ independent of everything else. 
Combining this with \eqref{eq:lr}, \eqref{eq:stein} and \eqref{aht1}--\eqref{aht2}, we obtain
\begin{align*}
ET_t\tilde{h}(W)
\leq \frac{1}{\sigma_*}E\left(\max_{1\leq j\leq d} |R_j|\right)\sqrt{\log d}+\frac{1}{\sigma_*^2}\Delta_1(|\log t|\vee1)\log d
+|E[\Xi]|.
\end{align*}
To estimate $|E[\Xi]|$, we rewrite it as follows. By exchangeability we have
\begin{align*}
&E[(\Lambda^{-1}D)_jD_kD_lU\partial_{jkl}\psi_t(W+(1-U)D)]\\
&=-E[(\Lambda^{-1}D)_jD_kD_lU\partial_{jkl}\psi_t(W'-(1-U)D)]\\
&=-E[(\Lambda^{-1}D)_jD_kD_lU\partial_{jkl}\psi_t(W+UD)].
\end{align*}
Hence we obtain
\begin{align}
E[\Xi]&=\frac{1}{4}\sum_{j,k,l=1}^dE[(\Lambda^{-1}D)_jD_kD_lU\{\partial_{jkl}\psi_t(W+(1-U)D)-\partial_{jkl}\psi_t(W+UD)\}]\nonumber\\
&=\frac{1}{4}\sum_{j,k,l,m=1}^dE[(\Lambda^{-1}D)_jD_kD_lD_mU(1-2U)\partial_{jklm}\psi_t(W+\tilde D)],\label{eq:xi}
\end{align}
where $\tilde D:=UD+V(1-2U)D$ and $V$ is a uniform random variable on $[0,1]$ independent of everything else. 
Note that $|U+V(1-2U)|\leq U\vee(1-U)\leq1$ and thus $\|\tilde D\|_\infty\leq \|D\|_{\infty}$. 

Now, note that $\frac{e^{-s}\eta}{\sigma_*\sqrt{1-e^{-2s}}}\leq1/\sqrt{2\log d}$ for every $s\geq t$ by assumption.
Hence, \eqref{eq:stein-sol} and Lemma \ref{l2} imply 
\[
\sum_{j,k,l,m=1}^d\sup_{y\in R(0;\eta)}\left|\partial_{jklm}\psi_t(x+y)\right|
\leq C\sigma_*^{-4}\frac{1}{t}\log^2 d.
\]
Combining this with \eqref{eq:xi} and $\|\tilde D\|_\infty\leq\|D\|_\infty$, we obtain
\begin{align*}
|\E\Xi|&\leq \frac{1}{4}\sum_{j,k,l,m=1}^d\E\left[|(\Lambda^{-1}D)_jD_kD_lD_m|\sup_{y\in R(0;\eta)}|\partial_{jklm}\psi_t(W+y) |;\|\tilde D\|_{\infty}\leq\eta\right]\\
&\quad+\frac{1}{4}\sum_{j,k,l,m=1}^d\E[|(\Lambda^{-1}D)_jD_kD_lD_m\partial_{jklm}\psi_t(W+\tilde D)|;\|D\|_{\infty}>\eta]\\
&\leq C\sigma_*^{-4}\frac{\log^2d}{t}(\Delta_2+\Delta_3(\eta)).
\end{align*}
Now the desired result follows from \eqref{3}.

\subsection{Proof of Theorem~\ref{ft4}}
Without loss of generality, we may assume $t<1$; otherwise, the theorem is trivial because $\sup_{h=1_A: A\in \mathcal{R}} |E h(W)- E h(Z)|\leq 1$.
We follow the proof of Theorem~\ref{t2} and bound $E(\Delta \psi_t(W)- W\cdot \nabla \psi_t(W))$.
From the independence of $X'$ and $X$ and the assumption that $E(W)=0$ and using the telescoping sum, we have
\bes{
E W \cdot \nabla \psi_t(W)=&E(W-W^{\{1:n\}})\cdot \nabla \psi_t(W)\\
=& \sum_{i=1}^n E(W^{\{1:(i-1)\}}-W^{\{1:i\}} )\cdot \nabla \psi_t(W).
}
Exchanging $X_i$ with $X_i'$, we have
\be{
E(W^{\{1:(i-1)\}}-W^{\{1:i\}} )\cdot \nabla \psi_t(W) =E(W^{\{1:i\}}-W^{\{1:(i-1)\}} )\cdot \nabla \psi_t(W^{\{i\}}).
}
Therefore,
\bes{
&E W \cdot \nabla \psi_t(W)\\
=&
\frac{1}{2}\sum_{i=1}^n E (W^{\{1:i\}}-W^{\{1:(i-1)\}} )\cdot (\nabla \psi_t(W^{\{i\}})-\nabla \psi_t(W))\\
=& \frac{1}{2}\sum_{i=1}^n \sum_{j,k=1}^d E (W^{\{1:i\}}-W^{\{1:(i-1)\}} )_j (W^{\{i\}}-W)_k \partial_{jk} \psi_t(W)\\
&+\frac{1}{2}\sum_{i=1}^n \sum_{j,k,l=1}^d E (W^{\{1:i\}}-W^{\{1:(i-1)\}} )_j (W^{\{i\}}-W)_k U (W^{\{i\}}-W)_l \partial_{jkl} \psi_t(W+UV( W^{\{i\}}-W)),
}
where $U, V$ are independent uniform random variables on $[0,1]$ and independent of everything else. 
Exchanging $X_i$ with $X_i'$ gives
\bes{
&E (W^{\{1:i\}}-W^{\{1:(i-1)\}} )_j (W^{\{i\}}-W)_k U (W^{\{i\}}-W)_l \partial_{jkl} \psi_t(W+UV( W^{\{i\}}-W))\\
=&-E (W^{\{1:i\}}-W^{\{1:(i-1)\}} )_j (W^{\{i\}}-W)_k U (W^{\{i\}}-W)_l \partial_{jkl} \psi_t(W-UV( W^{\{i\}}-W)).
}
Following similar arguments as in the proof of Theorem~\ref{t2}, we obtain the desired result.

\subsection{Proof of Corollary \ref{c2}}

Without loss of generality, we may assume
\begin{equation}\label{wlog}
(4\sqrt5)^6\frac{B_n^2\log^4(dn)}{\sigma_*^4 n}\leq 1.
\end{equation}
We prove the assertion in three steps. 
In Steps 1 and 2, we truncate the random variables and show that the error introduced by the truncation is negligible. In Step 3, we apply Theorem \ref{ft4} to the truncated variable.
\medskip

\noindent\textbf{Step 1.} 
Set $\kappa_n:=B_n\sqrt{5\log (dn)}$. For $i=1,\dots,n$ and $j=1,\dots,d$, define 
\[
\wt{X}_{ij}:=X_{ij}1_{\{|X_{ij}|\leq\kappa_n\}}-EX_{ij}1_{\{|X_{ij}|\leq\kappa_n\}},
\] 
and set $\wt{X}:=(\wt{X}_i)_{i=1}^n$ with $\wt{X}_i=(\wt{X}_{i1},\dots,\wt{X}_{id})^\top$. 
Note that $\max_{i,j}|\wt{X}_{ij}|\leq2\kappa_n$. Also, since $P(X_{ij}^2>x)\leq2e^{-x/B_n^2}$ for all $x\geq0$, Lemma 5.4 in \cite{Koike19} yields
\begin{equation}\label{trunc-l2}
EX_{ij}^21_{\{X_{ij}^2>\kappa_n^2\}}
\leq Ce^{-\kappa_n^2/B_n^2}\kappa_n^2
\leq C\frac{B_n^2\log(dn)}{n^5}.
\end{equation}

\noindent\textbf{Step 2.} 
Let $\wt{W}:=n^{-1/2}\sum_{i=1}^n\wt{X}_i$. On the event $\max_{1\leq i\leq n}\|X_i\|_\infty\leq\kappa_n$, we have
\ba{
|W_j-\wt{W}_j|
=\left|\frac{1}{\sqrt{n}}\sum_{i=1}^nEX_{ij}1_{\{|X_{ij}|>\kappa_n\}}\right|
\leq\frac{1}{\sqrt{n}}\sum_{i=1}^n\sqrt{EX_{ij}^21_{\{X_{ij}^2>\kappa_n^2\}}}
\leq C\frac{B_n\sqrt{\log(dn)}}{n^2}.
}
Therefore, Lemma 6.1 in \cite{CCKK19} yields
\begin{equation*}
P\left(\|W-\wt{W}\|_\infty>C\frac{B_n\sqrt{\log(dn)}}{n^2}\right)
\leq P\left(\max_{1\leq i\leq n}\|X_i\|_\infty>\kappa_n\right)\leq\frac{1}{2n^4}.
\end{equation*}
From this estimate and the Gaussian anti-concentration inequality, we obtain
\[
\sup_{h=1_A: A\in \mathcal{R}} |E h(W)- E h(Z)|
\leq C\left(\frac{1}{2n^4}+\frac{B_n \log d \sqrt{\log(dn)}}{n^2}+\wt\delta\right),
\]
where
\[
\wt\delta:=\sup_{h=1_A: A\in \mathcal{R}} |E h(\wt W)- E h(Z)|.
\]
Therefore, the proof is completed once we show
\[
\wt\delta\leq C\left(\frac{B_n^2\log^4 (dn)}{\sigma_*^4 n}\right)^{1/3}.
\]

\noindent\textbf{Step 3.} 
We apply Theorem \ref{ft4} to $\wt{X}$ and $\wt{W}$ with $\eta:=4\kappa_n/\sqrt{n}$. By construction we have $\delta_3(\eta)=0$. 
Meanwhile, we have
\ba{
&E\left[\max_{1\leq j,k\leq d}\left|\frac{1}{2n}\sum_{i=1}^n\left((\wt X'_{ij}-\wt X_{ij})(\wt X'_{ik}-\wt X_{ik})- E[(\wt X'_{ij}-\wt X_{ij})(\wt X'_{ik}-\wt X_{ik})]\right)\right|\right]\\
&\leq C\frac{\sqrt{\log d}}{n}E\left[\sqrt{\max_{1\leq j,k\leq d}\sum_{i=1}^n (X'_{ij}-X_{ij})^2(X'_{ik}-X_{ik})^2}\right]
\leq C\sqrt{\delta_2\log d},
}
where the first inequality follows from Nemirovski's inequality:
\begin{lemma}[Lemma 14.24 in \cite{BuVdG11}]
Let $Y_i$ be independent random variables taking values in a measurable space $\mathcal{Y}$ and let $\gamma_1, \dots, \gamma_p$ be real-valued measurable functions on $\mathcal{Y}$ such that $E\gamma_j (Y_i)$ exists.
For $m\geq 1$ and $p\geq e^{m-1}$, we have
\be{
E \max_{1\leq j\leq p} \left| \sum_{i=1}^n \left(\gamma_j(Y_i) -E \gamma_j (Y_i) \right) \right|^m\leq
\left[ 8 \log(2p) \right]^{m/2} E\left[\max_{1\leq j\leq p} \sum_{i=1}^n \gamma_j^2 (Y_i) \right]^{m/2}.
}
\end{lemma}
\noindent and the second one follows from the Schwarz inequality. 
We also have
\[
\frac{1}{2n}\sum_{i=1}^n E[(\wt X'_{ij}-\wt X_{ij})(\wt X'_{ik}-\wt X_{ik})]=\frac{1}{n}\sum_{i=1}^nE[\wt X_{ij}\wt X_{ik}]
\] 
and $\Sigma_{jk}=\frac{1}{n}\sum_{i=1}^nE[X_{ij}X_{ik}]$. 
Therefore, noting $X_{ij}-\wt{X}_{ij}=X_{ij}1_{\{|X_{ij}|>\kappa_n\}}-EX_{ij}1_{\{|X_{ij}|>\kappa_n\}}$, 
the Schwarz inequality and \eqref{trunc-l2} imply that
\ba{
&\max_{1\leq j,k\leq d}\left|\Sigma_{jk}-\frac{1}{2n}\sum_{i=1}^n E[(\wt X'_{ij}-\wt X_{ij})(\wt X'_{ik}-\wt X_{ik})]\right|\\
&\leq\max_{1\leq j,k\leq d}\max_{1\leq i\leq n}\left(\sqrt{E(X_{ij}-\wt{X}_{ij})^2E\wt{X}_{ik}^2}+\sqrt{EX_{ij}^2E(X_{ik}-\wt{X}_{ik})^2}\right)\\
&\leq C\frac{B_n^2\sqrt{\log (dn)}}{n^{5/2}}
\leq C\frac{B_n^2\log(dn)}{n}.
}
Consequently, we obtain
\ba{
\delta_1
\leq C\left(\sqrt{\delta_2\log d}
+\frac{B_n^2\log (dn)}{n}\right).
}
Moreover, Lemma 9 in \cite{CCK15} yields
\ba{
\delta_2
&\leq \frac{C}{n^2}\left\{\max_{1\leq j\leq d}E\left[\sum_{i=1}^n(\wt X'_{ij}-\wt X_{ij})^4\right]
+E\left[\max_{1\leq i\leq n}\max_{1\leq j\leq d}(\wt X'_{ij}-\wt X_{ij})^4\right]\log d\right\}\\
&\leq C\left(\frac{B_n^2}{n}
+\frac{B_n^4\log^3(dn)}{n^2}\right)
\leq C\frac{B_n^2}{n},
}
where the last inequality follows from \eqref{wlog}. 
Therefore, for any $t>0$ satisfying $\eta/\sqrt{t}\leq \sigma_*/\sqrt{\log d}$, we have
\ba{
\wt\delta
\leq C\left(\frac{1}{\sigma_*^2}\sqrt{\frac{B_n^2\log^3(dn)}{n}}(|\log t|\vee1)+\frac{B_n^2\log^2d}{\sigma_*^4 n}\frac{1}{t}+\sqrt{t}\log d\right).
}
Now let $t=(B_n^2\log(dn)/\sigma_*^4 n)^{2/3}$. 
Then we have
\[
\frac{\eta}{\sqrt{t}}\sqrt{\log d}
\leq 4\sqrt5\cdot\frac{B_n\log(dn)}{\sqrt n}\cdot\frac{n^{1/3} \sigma_*^{4/3}}{B_n^{2/3}\log^{1/3}(dn)}
=4\sqrt5 \left(\frac{B_n^2\log^4(dn)}{\sigma_*^4 n}\right)^{1/6}\sigma_*^2\leq \sigma_*
\] 
by \eqref{wlog}. So we can apply the above estimate with this $t$ and obtain
\begin{align*}
\wt\delta
&\leq C\left\{ 
\sqrt{\frac{B_n^2\log^3(dn)}{\sigma_*^4 n}}\left|\log\frac{B_n^2\log(dn)}{\sigma_*^4 n}\right|
+ \left(\frac{B_n^2\log^4 (dn)}{\sigma_*^4 n}\right)^{1/3} 
\right\}\\
&\leq C\left(\frac{B_n^2\log^4 (dn)}{\sigma_*^4 n}\right)^{1/3}, 
\end{align*}
where the last line follows from the inequality $|\log x|\leq C/x^{1/6}$ for $0<x\leq1$. 
 

\subsection{Proof of Theorem~\ref{ft5}}
Without loss of generality, we assume that the right-hand side of \eq{f16} is finite.
Let 
\be{
Y_i=\sum_{i'\in A_i} X_{i'},\quad Y_{ii'}=\sum_{i''\in A_{ii'} }X_{i''}.
}
From the independence assumption and $E X_i=0$, we have, with $U$ being a uniform distribution on $[0,1]$ and independent of everything else,
\bes{
E W \cdot \nabla \psi_t(W)=& \sum_{i=1}^n E X_i \cdot (\nabla \psi_t(W)-\nabla \psi_t(W-Y_i))\\
=& \sum_{i=1}^n \sum_{i'\in A_i} \sum_{j,k=1}^d E X_{ij} X_{i'k}  \partial_{jk} \psi_t(W-U Y_i)    \\
=& \sum_{i=1}^n \sum_{i'\in A_i} \sum_{j,k=1}^d E X_{ij} X_{i'k} \big[ \partial_{jk} \psi_t(W-U Y_i)  -\partial_{jk} \psi_t(W-Y_{ii'})  \big]\\
&+\sum_{i=1}^n \sum_{i'\in A_i} \sum_{j,k=1}^d E X_{ij} X_{i'k} E\partial_{jk} \psi_t(W-Y_{ii'}).
}
Because 
\be{
\sum_{i=1}^n \sum_{i'\in A_i}  E X_{ij} X_{i'k} E \partial_{jk} \psi_t(W)=\Sigma_{jk} E\partial_{jk} \psi_t(W),
}
we have by \eqref{aht3}
\bes{
&|E T_t \tilde h(W)|=|E[\langle \Sigma, \text{Hess} \psi_t(W)\rangle_{H.S.} -W\cdot \nabla \psi_t(W)]|\\
\leq&  \frac{C}{\sigma_*^3\sqrt{t}} (\log d)^{3/2} \sum_{i=1}^n \sum_{i'\in A_i}\sum_{i''\in A_{ii'}} E\big[ \max_{1\leq j,k,l\leq d}  (|X_{ij}X_{i'k}X_{i''l}|+|X_{ij}X_{i'k}|E|X_{i''l}|)\big].
}
Optimizing $t$ gives the desired bound.

\appendix

\section{Appendix}

\subsection{Proof of Proposition~\ref{p1}}\label{sec3-p1}

It suffices to show that there is a sequence $(x_n)_{n=1}^\infty$ of real numbers such that
\[
\rho:=\limsup_{n\to\infty}\sqrt{\frac{n}{\log^3d}}\left|P\left(\max_{1\leq j\leq d}W_j\leq x_n\right)-P\left(\max_{1\leq j\leq d}Z_j\leq x_n\right)\right|>0.
\]
We denote by $\phi_1$ and $\Phi_1$ the density and distribution function of the standard normal distribution, respectively.
For every $n$, we define $x_n\in\mathbb{R}$ as the solution of the equation $\Phi_1(x)^{d}=e^{-1}$, i.e.~$x_n:=\Phi_1^{-1}(e^{-1/d})$. 
Then we have $x_n/\sqrt{2\log d}\to1$ and $d(1-\Phi_1(x_n))\to1$ as $n\to\infty$ (cf.~the proof of Proposition 2.1 in \cite{Koike19}). 
Applying Theorem 1 in \cite{AGG89} with $I=\{1,\dots,d\}$, $B_\alpha=\{\alpha\}$ and $X_\alpha=1_{\{W_\alpha>x_n\}}$, we obtain
\begin{align*}
\left|P\left(\max_{1\leq j\leq d}W_j\leq x_n\right)-e^{-\lambda_n}\right|
\leq dP\left(W_1>x_n\right)^2,
\end{align*}
where $\lambda_n:=dP\left(W_1>x_n\right)$. 
By an analogous argument we also obtain
\[
\left|P\left(\max_{1\leq j\leq d}Z_j\leq x_n\right)-e^{-d(1-\Phi_{1}(x_n))}\right|
\leq d(1-\Phi_1(x_n))^2.
\]
Hence we have
\begin{align*}
&\left|P\left(\max_{1\leq j\leq d}W_j\leq x_n\right)-P\left(\max_{1\leq j\leq d}Z_j\leq x_n\right)\right|\\
&\geq|e^{-\lambda_n}-e^{-d(1-\Phi_1(x_n))}|-dP\left(W_1>x_n\right)^2-d(1-\Phi_1(x_n))^2.
\end{align*}

Now, since $x_n=O(\sqrt{\log d})=o(n^{1/6})$ by assumption, Theorem 1 in \cite[Chapter VIII]{Petrov75} (see also Eq.(2.41) in \cite[Chapter VIII]{Petrov75}) implies
\begin{equation}\label{eq:mdp}
\frac{P\left( W_1>x_n\right)}{1-\Phi_1(x_n)}=\exp\left(\frac{\gamma}{6\sqrt{n}}x_n^3\right)+O\left(\frac{x_n+1}{\sqrt{n}}\right).
\end{equation}
In particular, since $d(1-\Phi_1(x_n))\to1$, we have $dP(W_1>x)^2=O(d^{-1})$ and $d(1-\Phi_1(x_n))^2=O(d^{-1})$. Thus we obtain
\begin{equation}\label{p1-aim1}
\rho\geq\limsup_{n\to\infty}\sqrt{\frac{n}{\log^3d}}|e^{-\lambda_n}-e^{-d(1-\Phi_1(x_n))}|
\end{equation}
because $d^{-1}=o(n^{-1}\log^3d)$ by assumption. 
Moreover, using the Taylor expansion of the exponential function around 0, we deduce from \eqref{eq:mdp} 
\begin{equation*}
\lambda_n=d(1-\Phi_1(x_n))+\frac{\gamma}{6\sqrt{n}}x_n^3+o\left(\frac{x_n^3}{\sqrt{n}}\right)
\end{equation*}
and
\begin{align*}
e^{-\{\lambda_n-d(1-\Phi_1(x_n))\}}
=1-\frac{\gamma}{6\sqrt{n}}x_n^3+o\left(\frac{x_n^3}{\sqrt{n}}\right).
\end{align*}
Therefore, by \eqref{p1-aim1} we conclude that
\[
\rho\geq e^{-1}\frac{\sqrt{2}|\gamma|}{3}
\]
because $x_n/\sqrt{2\log d}\to1$. This completes the proof. 


\subsection{Proof of Lemma \ref{L2}}\label{sec3-l2}


First we introduce some notation. 
We denote by $\phi_1$ and $\Phi_1$ the density and distribution function of the standard normal distribution, respectively. We set $\bar\phi_1(u):=\phi_1(u)/\Phi_1(u)$. 
Obviously, $\bar\phi_1$ is strictly decreasing on $[0,\infty)$.

For a non-negative integer $\nu$, the $\nu$-th Hermite polynomial is denoted by $H_\nu$: $H_\nu(u)=(-1)^\nu\phi_1(u)^{-1}\phi_1^{(\nu)}(u)$.  
When $\nu\geq1$, we define the functions $h_\nu$ and $\bar h_\nu$ on $\mathbb{R}$ by 
\[
h_\nu(u)=H_{\nu-1}(u)\phi_1(u),\qquad
\bar h_{\nu}(u)=h_{\nu}(u)/\Phi_1(u)=H_{\nu-1}(u)\bar\phi_1(u)\qquad(u\in\mathbb{R}). 
\]
A simple computation shows 
\begin{equation}\label{eq:deriv-h}
h_\nu'(u)=-h_{\nu+1}(u),\qquad
\bar h_\nu'(u)=-\{H_{\nu}(u)+\bar h_\nu(u)\}\bar\phi_1(u).
\end{equation}
Also, we define the functions $\lambda$ and $\Lambda$ on $[0,\infty)$ by
\[
\lambda(u)=\frac{\phi_1(u)}{\phi_1(u+2\eta)}=e^{2u\eta+2\eta^2},\qquad
\Lambda(u)=\frac{\Phi_1(u)}{\Phi_1(u+2\eta)}\qquad(u\in[0,\infty)).
\]
A simple computation shows
\begin{equation}\label{eq:deriv-lambda}
\Lambda'(u)=\Lambda(u)\{\bar\phi_1(u)-\bar\phi_1(u+2\eta)\}.
\end{equation}
In particular, $\Lambda$ is non-decreasing on $[0,\infty)$. 

To extend the proof for the case with $r=1$ and $\eta=0$ to the general case, we need to deduce a bound analogous to \eqref{AHT-est}. Roughly speaking, we need to replace $\phi_1$ in the middle equation of \eqref{AHT-est} by $h_\nu$ to accomplish this. In the derivation of \eqref{AHT-est}, it plays a crucial role that $\phi_1$ is decreasing on $[0,\infty)$. 
However, $h_\nu$ does not have this property in general, so we will dominate it by an appropriate decreasing function to proceed analogously to the derivation of \eqref{AHT-est}. 
For this purpose, we need to introduce some additional notation. We denote by $u_\nu$ the maximum root of $H_\nu$. For example, $u_1=0,u_2=1,u_3=\sqrt{3}$. It is evident that $H_\nu$ is positive and strictly increasing on $(u_\nu,\infty)$. 
We also have $u_1<u_2<\cdots$ (see e.g.~\cite[Theorem 3.3.2]{Sz39}). 
Finally, set $M_{\nu}:=\max_{0\leq u\leq u_\nu}|H_{\nu-1}(u)|<\infty$ and define the function $\wt h_\nu:[0,\infty)\to(0,\infty)$ by
\[
\wt h_\nu(u)=M_{\nu}\phi_1(u)1_{[0,u_\nu]}(u)+h_{\nu}(u)1_{(u_\nu,\infty)}(u)\qquad (u\in[0,\infty)).
\]

%
\begin{lemma}\label{l2-1}
$\tilde h_\nu$ is decreasing on $[0,\infty)$ and $|h_{\nu}(u)|\leq \tilde h_\nu(|u|)$ for all $u\in\mathbb{R}$. 
\end{lemma}

\begin{proof}
Note that $h_\nu'(u)<0$ when $u>u_\nu$. Then, $\tilde h_\nu$ is evidently decreasing on $[0,\infty)$ by construction. The latter claim is also obvious by construction. 
\end{proof}

%

We will also need a counterpart of the latter part of the proof for the case with $r=1$ and $\eta=0$. The subsequent two lemmas will be used for this purpose.
\begin{lemma}\label{l2-2}
The function $u\mapsto H_\nu(u)\lambda(u)+\bar h_\nu(u)$ is strictly increasing on $[u_\nu,\infty)$. 
\end{lemma}
\begin{proof}
Since $H_\nu(u)\lambda(u)+\bar h_\nu(u)=H_\nu(u)\{\lambda(u)-1\}+\{H_\nu(u)+\bar h_\nu(u)\}$ and the function $u\mapsto H_\nu(u)\{\lambda(u)-1\}$ is non-decreasing on $[u_\nu,\infty)$, it suffices to prove $g:=H_\nu+\bar h_\nu$ is strictly increasing on $[u_\nu,\infty)$.
We have
\begin{align*}
g'(u)&=H'_{\nu}(u)-\{H_{\nu}(u)+\bar h_\nu(u)\}\bar\phi_1(u)\\
&=\Phi_1(u)^{-1}\{H'_{\nu}(u)\Phi_1(u)-h_{\nu+1}(u)-\bar h_{\nu}(u)\phi_1(u)\}.
\end{align*}
So we complete the proof once we show $g_1(u):=H'_{\nu}(u)\Phi_1(u)-h_{\nu+1}(u)-\bar h_{\nu}(u)\phi_1(u)>0$ for all $u> u_\nu$. We have
\begin{align*}
g_1'(u)&=H''_{\nu}(u)\Phi_1(u)+H'_{\nu}(u)\phi_1(u)+h_{\nu+2}(u)-\bar h_{\nu}'(u)\phi_1(u)+u\bar h_{\nu}(u)\phi_1(u)\\
&=\nu(\nu-1)H_{\nu-2}(u)\Phi_1(u)+u H_{\nu}(u)\phi_1(u)-\bar h_{\nu}'(u)\phi_1(u)+u\bar h_{\nu}(u)\phi_1(u),
\end{align*}
where the identity $H_{\nu+1}(u)=uH_\nu(u)-H_\nu'(u)$ is used to deduce the last line. Since $H_k(u)>0$ for $k\leq\nu$ and $u> u_\nu$, we have $g_1'(u)>0$ for $u>u_\nu$. Thus 
\[
g_1(u)>g_1(u_\nu)
=H_{\nu-1}(u_\nu)\{\nu\Phi_1(u_\nu)-\phi_1(u_\nu)\bar\phi_1(u_\nu)\}
\geq1/2-1/\pi>0
\] 
for all $u>u_\nu$. 
\end{proof}

Define the functions $F_\nu$ and $G$ on $\mathbb{R}^d$ by
\[
F_\nu(x)=\sum_{j=1}^d\bar h_{\nu}(x_{j})\Lambda(x_j),\qquad
G(x)=\prod_{k=1}^d\Phi_1(x_k+2\eta)\qquad (x\in\mathbb{R}^d).
\]

\begin{lemma}\label{l2-3}
For any $\beta>0$, 
\[
\sup_{x\in[u_\nu,\infty)^d}F_\nu(x)^\beta G(x)
=O((\log d)^{\beta\nu/2})
\]
as $d\to\infty$. 
\end{lemma}

\begin{proof}
Define the function $f$ on $\mathbb{R}^d$ by $f(x)=F_\nu(x)^\beta G(x)$, $x\in\mathbb{R}^d$. 
First we prove $f$ has a maximizer on $[u_\nu,\infty)^d$. 
Using \eqref{eq:deriv-h}--\eqref{eq:deriv-lambda}, we obtain
\ba{
\partial_lf(x)&=\beta F_\nu(x)^{\beta-1}\{\bar h_\nu'(x_l)\Lambda(x_l)+\bar h_\nu(x_l)\Lambda'(x_l)\}G(x)
+F_\nu(x)^\beta\bar\phi_1(x_l+2\eta)G(x)\\
&=\left[-\beta\left\{H_{\nu}(x_l)\lambda(x_l)+\bar h_\nu(x_l)\right\}\Lambda(x_l)
+F_\nu(x)\right]\bar\phi_1(x_l+2\eta)F_\nu(x)^{\beta-1}G(x).
}
Now, since $\beta\left\{ H_{\nu}(x)\lambda(x)+\bar h_\nu(x) \right\} \Lambda(x)\to\infty$ as $x\to\infty$ while $\sup_{x\in[u_\nu,\infty)^d}F_\nu(x)<\infty$, there is a number $\bar u\geq u_\nu$ such that for all $x$ with $x_l\geq \bar u$, 
$\partial_lf(x)<0$. 
This means $\sup_{x\in[u_\nu,\infty)^d}f(x)=\sup_{x\in[u_\nu,\bar u]^d}f(x)$ and thus $f$ has a maximizer on $[u_\nu,\infty)^d$.  

Let $x^*$ be a maximizer of $f$ on $[u_\nu,\infty)^d$. 
Then the proof is completed once we show $f(x^*)=O((\log d)^{\beta\nu/2})$ as $d\to\infty$. 
Let $m$ be the number of components in $x^*$ greater than $u_\nu$. 
If $m\leq\beta$, we have 
\[
f(x^*)\leq\left\{d\sup_{u\in[u_\nu,\infty)}\bar h_\nu(u)\Lambda(u)\right\}^\beta\Phi_1(u_\nu+2\eta)^{d-\beta}=o(1)
\]
as $d\to\infty$, so it suffices to consider the case $m>\beta$. 

Since $f$ is symmetric, we may assume $x^*_1,\dots,x^*_m>u_\nu$ and $x^*_{m+1}=\cdots=x^*_d=u_\nu$ without loss of generality. 
Then, for every $l=1,\dots,m$, we must have $\partial_lf(x^*)=0$. 
Thus we obtain
\[
\{H_{\nu}(x^*_1)\lambda(x^*_1)+\bar h_\nu(x^*_1)\}\Lambda(x^*_1)
=\cdots
=\{H_{\nu}(x^*_m)\lambda(x^*_m)+\bar h_\nu(x^*_m)\}\Lambda(x^*_m)
=\beta^{-1}F_\nu(x^*).
\]
Since $\Lambda$ is non-decreasing, the function $u\mapsto\beta\{H_{\nu}(u)\lambda(u)+\bar h_\nu(u)\}\Lambda(u)$ is strictly increasing on $[u_\nu,\infty)$ by Lemma \ref{l2-2}. 
Therefore, we have $x^*_1=\cdots=x^*_m=:u^*$ and hence 
\[
\beta\{H_{\nu}(u^*)\lambda(u^*)+\bar h_\nu(u^*)\}\Lambda(u^*)=F_\nu(x^*)=m\bar h_\nu(u^*)\Lambda(u^*)+(d-m)\bar h_\nu(u_\nu)\Lambda(u_\nu).
\] 
Now let $g_2(u):=\beta H_\nu(u)\lambda(u)\Lambda(u)\Phi_1(u)-(m-\beta) h_\nu(u)\Lambda(u)-(d-m)h_\nu(u_\nu)\Lambda(u_\nu)$. Then we have $g_2(u^*)=0$, $g_2(u_\nu)<0$ and $g_2(\sqrt{2\log d})\to\infty$ as $d\to\infty$ (note that $\lambda\geq1$ on $[0,\infty)$, $\Lambda(\sqrt{2\log d})\to1$ and $\eta=O(1)$ as $d\to\infty$). Since $g_2$ is increasing on $[u_\nu,\infty)$, we conclude $u^*=O(\sqrt{\log d})$ as $d\to\infty$. 
Consequently, we obtain $\lambda(u^*)=O(1)$ as $d\to\infty$ and thus
\[
f(x^*)\leq\beta^\beta\{H_{\nu}(u^*)\lambda(u^*)+\bar h_\nu(u^*)\}^\beta=O((\log d)^{\beta\nu/2})
\]
as $d\to\infty$. 
This completes the proof. 
\end{proof}

Now we are ready to prove Lemma \ref{L2}. 
\begin{proof}[Proof of Lemma \ref{L2}]
For every $q\in\{1,\dots,r\}$, set 
\ba{
\mathcal{N}_q(r)&:=\{(\nu_1,\dots,\nu_q)\in\mathbb{Z}^q:\nu_1,\dots,\nu_q\geq0,\nu_1+\cdots+\nu_q=r\},\\
\mathcal{J}_q(d)&:=\{(j_1,\dots,j_q)\in\{1,\dots,d\}^q:j_{p}\neq j_{p'}\text{ if }p\neq p'\}.
}
Then we have for all $A\in\mathcal{R}$
\begin{align*}
&\sum_{j_1,\dots,j_r=1}^d\sup_{y\in R(0;\eta)}\left|\int_A\partial_{j_1,\dots, j_r}\phi(z+y)dz\right|\\
&\leq C_r\sum_{q=1}^r\sum_{(\nu_1,\dots,\nu_q)\in\mathcal{N}_q(r)}\sum_{(j_1,\dots,j_q)\in\mathcal{J}_q(d)}\sup_{y\in R(0,\eta)}\left|\int_A\partial^{\nu_1}_{j_1}\cdots\partial^{\nu_q}_{j_q}\phi(z+y)dz\right|,
\end{align*}
where $C_r>0$ depends only on $r$. 
Therefore, we obtain the desired result once we prove
\[
\sup_{A\in\mathcal{R}}\sum_{(j_1,\dots,j_q)\in\mathcal{J}_q(d)}\sup_{y\in R(0;\eta)}\left|\int_A\partial^{\nu_1}_{j_1}\cdots\partial^{\nu_q}_{j_q}\phi(z+y)dz\right|
=O((\log d)^{r/2})
\qquad
\text{as }d\to\infty
\]
for any (fixed) $(\nu_1,\dots,\nu_q)\in\mathcal{N}_q(r)$ with $q\in\{1,\dots,r\}$. 

Take $A=\prod_{j=1}^d(a_j,b_j)\in\mathcal{R}$ arbitrarily and set
\[
I_A:=\sum_{(j_1,\dots,j_q)\in\mathcal{J}_q(d)}\sup_{y\in R(0;\eta)}\left|\int_A\partial^{\nu_1}_{j_1}\cdots\partial^{\nu_q}_{j_q}\phi(z+y)dz\right|.
\]
Then we have
\begin{align*}
I_A
&=\sum_{(j_1,\dots,j_q)\in\mathcal{J}_q(d)}\sup_{y\in R(0;\eta)}\left(\prod_{p=1}^q\left|h_{\nu_p}(b_{j_p}+y_{j_p})-h_{\nu_p}(a_{j_p}+y_{j_p})\right|\right)\prod_{k:k\neq j_1,\dots,j_q}\left\{\Phi_1(b_k+y_k)-\Phi_1(a_k+y_k)\right\}\\
&\leq\sum_{(j_1,\dots,j_q)\in\mathcal{J}_q(d)}\sup_{y\in R(0;\eta)}\left(\prod_{p=1}^q\left(|h_{\nu_p}(b_{j_p}+y_{j_p})|+|h_{\nu_p}(a_{j_p}+y_{j_p})|\right)\right)\\
&\hphantom{\leq\sum_{(j_1,\dots,j_q)\in\mathcal{J}_q(d)}\sup_{y\in R(0;\eta)}}
\times\prod_{k:k\neq j_1,\dots,j_q}\left\{\Phi_1(b_k+y_k)+\Phi_1(-a_k-y_k)-1\right\},
\end{align*}
where we use the identity $1-\Phi_1(x)=\Phi_1(-x)$ to deduce the last line. 
Set $c_j:=(|a_j|\wedge|b_j|)\vee\eta$, $j=1,\dots,d$. Then, we have $\min\{|a_j+y_j|,|b_j+y_j|\}\geq c_j-\eta\geq0$ for all $j$. 
Thus, noting that $\Phi$ is increasing and bounded by 1, we obtain by Lemma \ref{l2-1} 
\begin{align*}
I_A
&\leq2^q\sum_{(j_1,\dots,j_q)\in\mathcal{J}_q(d)}\left(\prod_{p=1}^q\tilde h_{\nu_p}(c_{j_p}-\eta)\right)\prod_{k:k\neq j_1,\dots,j_q}\Phi_1(c_k+\eta)\\
&=2^q\sum_{(j_1,\dots,j_q)\in\mathcal{J}_q(d)}\left(\prod_{p=1}^q\frac{\tilde h_{\nu_p}(c_{j_p}-\eta)}{\Phi_1(c_{j_p}-\eta)}\frac{\Phi_1(c_{j_p}-\eta)}{\Phi_1(c_{j_p}+\eta)}\right)\prod_{k=1}^d\Phi_1(c_k+\eta)\\
&\leq2^q\left(\prod_{p=1}^q\sum_{j_p=1}^d\frac{\tilde h_{\nu_p}(c_{j_p}-\eta)}{\Phi_1(c_{j_p}-\eta)}\frac{\Phi_1(c_{j_p}-\eta)}{\Phi_1(c_{j_p}+\eta)}\right)\prod_{k=1}^d\Phi_1(c_k+\eta).
\end{align*}
Now, since $\sum_p\nu_p=r$, the generalized AM-GM inequality yields 
\begin{align*}
I_A
&\leq2^q\sum_{p=1}^q\frac{\nu_p}{r}\left(\sum_{j_p=1}^d\frac{\tilde h_{\nu_p}(c_{j_p}-\eta)}{\Phi_1(c_{j_p}-\eta)}\frac{\Phi_1(c_{j_p}-\eta)}{\Phi_1(c_{j_p}+\eta)}\right)^{r/\nu_p}\prod_{k=1}^d\Phi_1(c_k+\eta)\\
&\leq C'_r\sum_{p=1}^q\left\{
\left(\sum_{j=1}^d\bar\phi_1(c_{j}-\eta)\Lambda(c_{j}-\eta)\right)^{r/\nu_p}\right.\\
&\left.\qquad\qquad+\left(\sum_{j=1}^d\bar h_{\nu_p}(c_{j}-\eta)\Lambda(c_{j}-\eta)1_{(u_{\nu_p},\infty)}(c_{j}-\eta)\right)^{r/\nu_p}\right\}\prod_{k=1}^d\Phi_1((c_k-\eta)+2\eta)\\
&\leq C'_r\sum_{p=1}^q\left\{
\sup_{x\in[0,\infty)^d}F_{1}(x)^{r/\nu_p}G(x)
+\sup_{x\in[u_{\nu_p},\infty)^d}F_{\nu_p}(x)^{r/\nu_p}G(x)
\right\},
\end{align*}
where $C'_r>0$ depends only on $r$ (note that $\bar h_\nu$ is positive on $[u_\nu,\infty)$). Consequently, by Lemma \ref{l2-3} we conclude 
$\sup_{A\in\mathcal{R}}I_A=O((\log d)^{r/2})$ as $d\to\infty$. 
\end{proof}

\subsection{Proof of Proposition \ref{p2}}
We follow the proof of Proposition 1.4 of \cite{Zh18}.
Let $A\in \mathcal{R}$ be a given hyperrectangle. For a parameter $\epsilon$ to be specified later, define
\be{
A^\epsilon=\left\{ x\in \mathbb{R}^d: \inf_{a\in A} |x-a|\leq \epsilon   \right\},
}
\be{
A_\epsilon=\left\{ x\in \mathbb{R}^d: \inf_{a\in \mathbb{R}^d\backslash A} |x-a|\geq \epsilon   \right\}.
}
Applying the Gaussian anti-concentration inequality in Lemma~\ref{l1} with $Y=(Z^\top, -Z^\top)^\top$ gives
\be{
P(Z\in A^\epsilon \backslash A)\leq C \epsilon \sqrt{\log d},\  \text{and}\ P(Z\in A \backslash A_\epsilon)\leq C \epsilon \sqrt{\log d}.
}
We may regard $T$ as being coupled to $Z$ so that $E[|T-Z|^2]=\mathcal{W}_2(T, Z)^2$.
Then
\bes{
P(T\in A)&\leq P(|T-Z|\leq \epsilon, T\in A)+P(|T-Z|>\epsilon)\\
&\leq P(Z\in A^\epsilon) +\epsilon^{-2} \mathcal{W}_2(T, Z)^2\\
&\leq P(Z\in A)+C \epsilon \sqrt{\log d} +\epsilon^{-2} \mathcal{W}_2(T, Z)^2.
}
Similarly, 
\bes{
P(Z\in A)&\leq P(Z\in A_\epsilon)+C\epsilon \sqrt{\log d}\\
&\leq P(|T-Z|\leq \epsilon, Z\in A_\epsilon)+P(|T-Z|>\epsilon)+C\epsilon \sqrt{\log d}\\
&\leq P(T\in A)+\epsilon^{-2} \mathcal{W}_2(T, Z)^2+C\epsilon \sqrt{\log d}.
}
Thus,
\be{
|P(T\in A)-P(Z\in A)|\leq \epsilon^{-2} \mathcal{W}_2(T, Z)^2+C\epsilon \sqrt{\log d},
}
and taking $\epsilon=\frac{W_2(T, Z)^{2/3}}{(\log d)^{1/6}}$ gives the result.

\section*{Acknowledgements}

We thank the two anonymous referees for their careful reading of the manuscript and for their valuable suggestions which led to many improvements.
Fang X. was partially supported by Hong Kong RGC ECS 24301617 and GRF 14302418 and 14304917, a CUHK direct grant and a CUHK start-up grant. 
Koike Y. was partially supported by JST CREST Grant Number JPMJCR14D7 and JSPS KAKENHI Grant Numbers JP17H01100, JP18H00836, JP19K13668.


\end{document}